\documentclass{amsproc}

\usepackage{amsmath}
\usepackage{amssymb}
\usepackage{rotating}
\usepackage{graphicx}
\DeclareGraphicsExtensions{.png,.pdf,.eps}
\usepackage[all,2cell]{xy}
\usepackage{tikz}

\newtheorem{theorem}{Theorem}[section]
\newtheorem{lemma}[theorem]{Lemma}
\newtheorem{corollary}[theorem]{Corollary}
\newtheorem{proposition}[theorem]{Proposition}
\newtheorem{conjecture}[theorem]{Conjecture}

\theoremstyle{definition}

\newtheorem{definition}[theorem]{Definition}
\newtheorem{notation}[theorem]{Notation}

\newtheorem{remark}[theorem]{Remark}

\def\C{{\mathbb C}}

\def\G{{\mathbb G}}

\def\P{{\mathbb P}}
\def\Q{{\mathbb Q}}
\def\R{{\mathbb R}}
\def\Z{{\mathbb Z}}

\def\cA{{\mathcal A}}

\def\cC{{\mathcal C}}
\def\cD{{\mathcal D}}
\def\cE{{\mathcal E}}
\def\cF{{\mathcal F}}

\def\cH{{\mathcal{H}}}

\def\cK{{\mathcal K}}

\def\cM{{\mathcal M}}

\def\cO{{\mathcal{O}}}

\def\cU{{\mathcal U}}

\def\Q{{\mathbb{Q}}}
\def\G{{\mathbb{G}}}

\def\fg{{\mathfrak g}}
\def\fh{{\mathfrak h}}

\def\fb{{\mathfrak b}}

\def\f{\varphi}

\def\f{\varphi}

\def\lra{\longrightarrow}

\def\lra{\longrightarrow}
\def\rat{\dashrightarrow}

\def\operatorname#1{\mathop{\rm #1}\nolimits}

\def\Hom{\operatorname{Hom}}

\def\Pic{\operatorname{Pic}}

\def\Spec{\operatorname{Spec}}

\def\deg{\operatorname{deg}}

\def\det{\operatorname{det}}

\def\rat{\operatorname{RatCurves}}

\def\qed{\hspace{\fill}$\rule{2mm}{2mm}$}
\def\NE{{\operatorname{NE}}}

\def\Nef{{\operatorname{Nef}}}
\def\Amp{{\operatorname{Amp}}}

\newcommand{\cNE}[1]{\overline{\NE}(#1)}

\def\Gl{\operatorname{GL}}

\def\len{{\lambda}}

\newcommand{\sgn}{\operatorname{sgn}}
\newcommand{\Chi}{\ensuremath \raisebox{2pt}{$\chi$}}

\DeclareMathOperator{\ch}{\mathrm{Chains}}

\makeindex

\begin{document}
%\pagewiselinenumbers

%%%%%%%%%%%%%%%%%%%%%%%%%%
%%% TITLES
\title[A geometric characterization of flag varieties]%[Fano manifolds whose elementary contractions are $\P^1$-fibrations]
{Fano manifolds whose elementary contractions are smooth $\P^1$-fibrations: a geometric characterization of flag varieties}

%\title[A geometric characterization of flag varieties]{A geometric characterization of flag varieties}
%%%%%%%%%%%%%%%%%%%%%%%%%%

\author[G. Occhetta]{Gianluca Occhetta}
\address{Dipartimento di Matematica, Universit\`a di Trento, via
Sommarive 14 I-38123 Povo di Trento (TN), Italy}
\thanks{First and second author were partially supported by PRIN project ``Geometria delle variet\`a algebriche'' and the Department of Mathematics of the University of Trento.}
\email{gianluca.occhetta@unitn.it}

\author[L.E. Sol\'a Conde]{Luis E. Sol\'a Conde}
\address{Dipartimento di Matematica, Universit\`a di Trento, via
Sommarive 14 I-38123 Povo di Trento (TN), Italy}
%\address{Korea Institute for Advanced Study, 85 Hoegiro, Dongdaemun-gu, Seoul, 130--722, Republic of Korea}
\thanks{Second author partially supported by the Spanish government grant MTM2009-06964, and by the Korean National Researcher Program 2010-0020413 of NRF}
\email{lesolac@gmail.com}

\author[K. Watanabe]{Kiwamu Watanabe}
\address{Course of Mathematics, Programs in Mathematics, Electronics and Informatics,
Graduate School of Science and Engineering, Saitama University.
Shimo-Okubo 255, Sakura-ku Saitama-shi, 338-8570 Japan}
\email{kwatanab@rimath.saitama-u.ac.jp}
\thanks{The third author was partially supported by JSPS KAKENHI Grant Number 24840008.}

\author[J. Wi\'sniewski]{Jaros\l{}aw A. Wi\'sniewski}
\address{Instytut Matematyki UW, Banacha 2, 02-097 Warszawa, Poland}
\email{J.Wisniewski@uw.edu.pl} \thanks{The fourth author was supported  by Polish National Science Center grant 2013/08/A/ST1/00804. A part of this work was done when he visited CIRM at the University of Trento.}


\begin{abstract}
  The present paper provides a geometric characterization of complete
  flag varieties for semisimple algebraic groups. Namely, if $X$ is a
  Fano manifold whose all elementary contractions are
  $\P^1$-fibrations then $X$ is isomorphic to the complete flag
  manifold $G/B$ where $G$ is a semi-simple Lie algebraic group and
  $B$ is a Borel subgroup of $G$.
\end{abstract}

\maketitle

\section{Introduction}

Rational homogeneous manifolds constitute one of the most important
class of examples of Fano manifolds. Whilst the representation theory
of semisimple Lie groups provides a powerful tool to describe many
different aspects of the geometry of this type of manifolds, they lack
good intrinsic characterization within the class of Fano manifolds.

One of the most natural conjectures in this direction is the one
proposed by F.~Campana and T.~Peternell in 1991 (cf. \cite[Special
case~11.2, Conjecture~11.1]{CP1}):

\begin{conjecture}[Campana--Peternell]
\label{conj:CPconj}
  Any Fano manifold whose tangent bundle is nef is rational homogeneous.
\end{conjecture}

We note that the above conjecture is in the spirit of the celebrated
Hartshorne--Frankel conjecture proved by Mori in \cite{Mo}, which
constitutes one of the cornerstones of the Minimal Model Program in
algebraic geometry.

The conjecture is known to be true in some special situations (\cite{CP2, Mk,
  Hw, SW, Wa2,P}). In the quoted references the proofs depend on detailed
classifications of the manifolds satisfying the required properties,
hence the general question of showing how the homogeneity follows from the nefness of the tangent
bundle is still an open problem.

Recently, a possible strategy towards the solution of this
problem has been proposed in \cite{MOSW}. The paper showed how the
classification of manifolds of Picard number two with two smooth $\P^1$-fibrations could be
used recursively to associate a finite Dynkin diagram with any manifold
with nef tangent bundle supporting as many smooth $\P^1$-fibrations as its Picard number,
identifying in this way a homogeneous model of $X$, to which it
is conjecturally isomorphic.
In \cite{MOSW} it was then proposed a method to prove homogeneity
by reconstructing the manifold using families of rational curves,
and this method was successfully applied to manifolds
with associated Dynkin diagram  of type ${\rm A}_n$.

In this paper we generalize (using completely different techniques) the results in \cite{MOSW}, associating
a finite Dynkin diagram $\cD$ and a homogeneous model $G/B$ with any Fano manifold $X$ whose elementary contractions are smooth $\P^1$-fibrations. We then prove the following

\begin{theorem}\label{conj:CPforFT}
Let $X$ be a Fano manifold whose elementary contractions are
smooth $\P^1$-fibrations. Then
$X$ is isomorphic to a complete flag manifold $G/B$, where $G$ is a
semisimple algebraic group and $B$ a Borel subgroup.
\end{theorem}

Our proof of this statement is based on the following ideas: every
smooth $\P^1$-fibration in $X$ provides an involution of the vector
space $N^1(X)$ of classes of $\R$-divisors in $X$. We show in Section
\ref{sec:reflect} that the set of these involutions generates a finite
reflection group, which is the Weyl group $W$ of a
semisimple Lie algebra $\fg$, defining a Dynkin diagram $\cD$ and a homogeneous model $G/B$ for $X$.
We can then consider an isomorphism $\psi: N^1(X)\to N^1(G/B)$ preserving
the relative canonical classes of the $\P^1$-fibrations of $X$ and $G/B$,
using it to prove that they have the same dimension and the same cohomology of line bundles (see Corollary \ref{cor:cohomequal}).

In Section \ref{sec:BSvar} we use the $\P^1$-fibrations in $X$ to
define a set of auxiliary manifolds called Bott-Samelson
varieties, which are analogues of the Bott-Samelson varieties that appear classically in the study of Schubert cycles of homogeneous manifolds. We then show in Section \ref{sec:unique} that the
recursive construction of appropriately chosen Bott-Samelson varieties depends only on the
combinatorics of the Weyl group $W$ (cf. Propositions \ref{prop:uniqueADE} and \ref{prop:uniqueBC}), and we infer from this the isomorphism between $X$ and its homogeneous model.

This proof works with an exception, which occurs when a connected component of $\cD$ is ${\rm F}_4$, so, to conclude the proof of Theorem \ref{conj:CPforFT}, we need to treat this case with different arguments, presented in Section \ref{sec:F4}.

Finally, in Section \ref{sec:appCP}, we briefly discuss how our
results merge into the general problem of Campana--Peternell
Conjecture.

%%%%%%%%%%%%%%%%%%%%%%%%%%%%%%%%%%%%%%%%%%%%%%%%%
%%\noindent{\bf Acknowledgements:} First and second author were partially supported by PRIN project ``Geometria delle variet\`a algebriche'' and the Department of Mathematics of the University of Trento. The second author was partially supported by the Spanish government grant MTM2009-06964, and by the Korean National Researcher Program 2010-0020413 of NRF. The third author was partially supported by JSPS KAKENHI Grant Number 24840008. The fourth author was supported  by Polish National Science Center grant 2013/08/A/ST1/00804. A part of this work was done when he visited CIRM at the University of Trento.
%%%%%%%%%%%%%%%%%%%%%%%%%%%%%%%%%%%%%%%%%%%%%%%%%
\begin{notation}\label{not:ftmanifold}
Along this paper, unless otherwise stated, $X$ will denote a Fano manifold of Picard number $\rho_X=n$ whose elementary contractions $\pi_i:X\to X_i$ are smooth $\P^1$-fibrations.
The vector space of $\R$-divisors on $X$ modulo numerical equivalence will be denoted by $N^1(X)$, and its dual, that is the vector space of real $1$-cycles modulo numerical equivalence, by $N_1(X)$. The closure of the cone generated by effective $1$-cycles, the so called Mori cone of $X$, will be denoted by $\cNE{X}\subset N_1(X)$.
For every $i$ we will denote by $\Gamma_i$ a fiber of $\pi_i$ and by $R_i$ the extremal ray of $\cNE{X}$ generated by the class of $\Gamma_i$. The relative canonical divisor of $\pi_i$ will be denoted by $K_i$.
\end{notation}

\section{Homogeneous models
}\label{sec:reflect}

%%\begin{notation}\label{not:ftmanifold}
%%Along this paper, unless otherwise stated, $X$ will denote a Fano manifold of Picard number $\rho_X=n$ whose elementary contractions $\pi_i:X\to X_i$ are smooth $\P^1$-fibrations.
%%The vector space of $\R$-divisors on $X$ modulo numerical equivalence will be denoted by $N^1(X)$, and its dual, that is the vector space of real $1$-cycles modulo numerical equivalence, by $N_1(X)$. The closure of the cone generated by effective $1$-cycles, the so called Mori cone of $X$, will be denoted by $\cNE{X}\subset N_1(X)$.
%%For every $i$ we will denote by $\Gamma_i$ a fiber of $\pi_i$ and by $R_i$ the extremal ray of $\cNE{X}$ generated by the class of $\Gamma_i$. The relative canonical divisor of $\pi_i$ will be denoted by $K_i$.
%%\end{notation}

In this section we will show how to associate a flag manifold $G/B$ with  a Fano manifold $X$  whose elementary contractions $\pi_i$, $i=1,\dots,n$, are smooth $\P^1$-fibrations.

\begin{notation}\label{notn:div}
Since $X$ is a Fano manifold, two divisors on $X$ are numerically equivalent if and only if they are linearly  equivalent; in particular we can identify the space $N^1(X)_\Z$ of divisors modulo numerical equivalence with
$\Pic(X)$. For this reason, with abuse of notation we will denote the numerical class of a divisor and the corresponding line bundle with the same letter, using additive notation.
\end{notation}

As a first observation, we note the following:

\begin{lemma}\label{lem:basis}
With the same notation as in \ref{not:ftmanifold}, the classes $[\Gamma_i]$, $i=1,\dots,n$ form a basis of the real vector space $N_1(X)$, and generate the Mori cone of $X$. In particular, $X$ has Picard number $n$.
\end{lemma}

\begin{proof}
Since every $\Gamma_i$ provides an unsplit dominating family of rational curves and  $\R_{\geq 0}[\Gamma_i]$ are extremal rays of $\cNE{X}$, the first part of the statement follows from a direct application of  \cite[Lemma~5.2]{A}. Since $X$ is a Fano manifold, its Mori cone is generated by the rays associated with its elementary contractions: this completes the proof.
\end{proof}

 We will discuss now some general facts about cohomology of divisors on smooth $\P^1$-fibrations, that we will need later on.

\subsection{Relative duality on $\P^1$-fibrations}\label{ssec:leray}

The following result may be understood as an extension to the case of smooth $\P^1$-fibrations of the Borel-Weil-Bott Theorem on $\P^1$. Using similar ideas, a short proof of the full Borel-Weil-Bott Theorem was given in \cite{Lu}.

\begin{lemma}\label{lem:leray}
Let $\pi:M\to Y$ be a smooth $\P^1$-fibration over a smooth manifold $Y$,  denote by $\Gamma$ one of its fibers and by $K$ its relative canonical divisor. Then for every Cartier divisor $D$ on $M$, setting $l:=D\cdot\Gamma$ and $\sgn(\alpha):=\alpha/|\alpha|$ for $\alpha\neq 0$,  $\sgn(0):=1$, one has
\begin{eqnarray}
H^{i}(M,\cO_M(D)) & \cong  &H^{i+\sgn(l+1)}(M,\cO_M(D+(l+1)K)), \mbox{ for every }i\in\Z.
\label{eq:leray}
\end{eqnarray}
In particular:
\begin{itemize}
\item $H^{i}(M,\cO_M(D))=0$ if $l=-1$, and
\item $\Chi(M,\cO_M(D))=-\Chi(M,\cO_M(D+(l+1)K))$ for every divisor $D$.
\end{itemize}
\end{lemma}

\begin{proof}
Note first that $l+1=D\cdot \Gamma+1=-((D+(l+1)K)\cdot\Gamma+1)$, so we may assume, without loss of generality, that $l\leq -1$. In this case, moreover, equation (\ref{eq:leray}) holds trivially for $i\leq -1$, so we may also assume that $i\geq 0$.

First of all we claim that there is a natural isomorphism:
$$
R^1\pi_*\cO_M(D)\cong\pi_*\cO_M(D+(l+1)K).
$$
In fact, relative Serre duality provides $R^1\pi_*\cO_M(D)\cong(\pi_*\cO_M(K-D))^\vee$, and the fact that the vector bundle on $Y$ appearing in the right hand side coincides with $\pi_*\cO_M(D+(l+1)K)$ follows from Lemma \ref{lem:dual} below.

Then, applying the Leray spectral sequence with respect to the morphism $\pi$ to both $\cO_M(D)$ and $\cO_M(D+(l+1)K)$, we obtain the required isomorphism: in the left hand side, since $l\leq -1$, the only nonzero elements of the sequence are those of the form $H^i(Y,R^1\pi_*\cO_M(D))$, hence there is an isomorphism $H^{i+1}(M,\cO_M(D))\cong H^{i}(Y,R^1\pi_*\cO_M(D))$; on the right hand side, since $(D+(l+1)K)\cdot \Gamma=-l-2\geq-1$, the only nonzero elements of the  sequence are those of the form $H^i(Y,\pi_*\cO_M(D+(l+1)K)))$, so that $ H^i(Y,\pi_*\cO_M(D+(l+1)K))\cong H^i(M,\cO_M(D+(l+1)K))$. Summing up we have:
\begin{eqnarray*}
H^{i+1}(M,\cO_M(D))\hspace{-0.2cm}&\cong&\hspace{-0.2cm} H^i(Y,R^1\pi_*\cO_M(D))\cong H^i(Y,\pi_*\cO_M(D+(l+1)K)))\\
&\cong& H^i(M,\cO_M(D+(l+1)K)).  \qquad\qquad \qquad
\end{eqnarray*}
\par
%\vspace{-0.6 cm}
%\qedhere
\qed
\end{proof}

\begin{lemma}\label{lem:dual}
Let $\pi:M\to Y$ be a smooth $\P^1$-fibration over a smooth manifold $Y$,  denote by $\Gamma$ one of its fibers and by $K$ its relative canonical divisor. Then for every divisor $D$ we have $$\big(\pi_*\cO_M(K-D)\big)^\vee\cong\pi_*\cO_M(D+(D\cdot \Gamma+1)K).$$
\end{lemma}

\begin{proof}
The statement is clearly true when $(K-D)\cdot \Gamma< 0$, so we may assume $l:=D\cdot \Gamma\leq -2$.

If $\pi:M\to Y$ is a $\P^1$-bundle or, equivalently, if $M$ admits a divisor $H$ satisfying $H\cdot \Gamma=1$, then denoting by $\cF$ the rank two vector bundle $\pi_*\cO_M(H)$, and writing $\cO_M(D)\cong\cO_M(l H)\otimes \pi^*(B)$ for some $B\in\Pic(Y)$, we have:
$$\begin{array}{l}\vspace{0.2cm}
\pi_*\cO_M(K-D)=S^{-(l+2)}\cF\otimes B^\vee \otimes\det(\cF)\\
\pi_*\cO_M(D+(l+1)K)=
S^{-(l+2)}\cF\otimes B\otimes\det(\cF)^{l+1}
\end{array}
$$
Since $\cF^\vee=\cF\otimes\det(\cF)^{\vee}$, it follows that
$$\big(S^{-(l+2)}(\cF)\big)^\vee=\big(S^{-(l+2)}(\cF)\big)\otimes\det(\cF)^{l+2}$$ which, together with our previous two equalities, concludes the proof.

We may then assume that $\pi:M\to Y$ is not a $\P^1$-bundle and, in particular, that $l=D\cdot \Gamma$ is even. Let us denote by $\cE$ the rank three vector bundle  $\pi_*\cO_M(-K)$;  the natural map $\pi^*(\cE)\to \cO_M(K)^{\vee}$ defines an inclusion $\iota:M\hookrightarrow P:=\P(\cE)$ which makes the following diagram, where   $p:P \to Y$ is the natural projection, commutative:
$$
\xymatrix{M \ar@{^{(}->}[]+<0ex,-1.7ex>;[d]_{\iota}
\ar[rrd]^{\pi}&&\\P\ar[rr]^{p}&&Y}
$$
Note that, by construction, $M\subset P$ is a divisor whose associated line bundle is $\cO_P(2)$, so that $M$ defines an everywhere nondegenerate (symmetric) isomorphism $\cE\cong\cE^\vee$.

We may write $\cO_M(D)\cong\cO_M(-(l/2)K)\otimes\pi^*B$ for some $B\in\Pic(Y)$, so that:
$$
\begin{array}{l}\vspace{0.2cm}
\pi_*\cO_M(K-D)=\pi_*\cO_M((l/2+1)K)\otimes B^\vee=S^{-(l/2+1)}\cE\otimes B^\vee\\
\pi_*\cO_M(D+(l+1)K)=
S^{-(l/2+1)}\cE\otimes B
\end{array}
$$
The proof is finished by using the isomorphism $\cE\cong\cE^\vee$.
\end{proof}

%%%%%%%%%%%%%%%%%%%%%%%%%%%%%%%%%%%%%%%%%%%%%%%%

\subsection{Reflection groups}\label{ssec:weyl}

We will now apply results of the previous subsection to the case of a Fano manifold $X$
satisfying Notation~\ref{not:ftmanifold}.

For every elementary contraction $\pi_i:X\to X_i$ we will consider the affine involution $r'_i$ of $N^1(X)$ given by:
$$
r'_i(D):=D+(D\cdot\Gamma_i+1)K_i. %,\quad D\in N^1(X),\,\,\,i=1,\dots n.
$$
Denoting by $T$ the translation by $K_X/2$ in $N^1(X)$, that is $T(D):=D+K_X/2$, every composition $r_i:=T^{-1}\circ r'_i\circ T$ is a linear involution of $N^1(X)$ given by
$$r_i(D)=D+ (D\cdot\Gamma_i)K_i, $$
which is a {\it reflection}, i.e. it is an order two automorphism that fixes a hyperplane in $N^1(X)$. In our case, $r_i(K_i)=-K_i$ and the fixed hyperplane is
\begin{equation}\label{eq:hyper}
\Gamma_i^{\perp}:=\{D|D\cdot\Gamma_i=0\}\subset N^1(X).
\end{equation}
We will denote by $W\subset\Gl(N^1(X))$ the group generated by the reflections $r_i$, $i=1,\dots, n$.
Via conjugation with $T$, the group of affinities $W'$ generated by the reflections $r'_i$ may be identified with $W$.
The following lemma shows how cohomology behaves with respect to the action of $W'$.

\begin{lemma}\label{lem:cohomW}
If $D \in \Pic(X)$ such that $w'(D)\in T(\Nef(X))$ for some $w'\in W'$ then there exists $i\in\Z$, such that
$$
H^j(X,D)=0,\mbox{ for }j\neq i.
$$
\end{lemma}
\begin{proof}
Note first that for $D\in T(\Nef(X))$, one has
$$
H^j(X,D)=H^j(X,K_X +D-K_X/2-K_X/2)=0,\mbox{ for }j\neq 0,
$$
where the last equality follows, being $D-K_X/2$ nef and $-K_X/2$ ample, by Kodaira Vanishing Theorem. Since every $w'\in W'$ is a composition of reflections $r_i'$, $i=1,\dots,n$, the result follows from recursive use of Lemma \ref{lem:leray}.
\end{proof}

\begin{definition}
Let $\Chi_X:\Pic(X)\to\Z$ be the function which associates with a line bundle its Euler characteristic, i.e. $\Chi_X(L)=\Chi(X,L)$; this function, as proved by Snapper (Cf. \cite[Section 1, Theorem]{Kl}), has the property that, given
$L_1, \dots, L_t \in \Pic(X)$,  $\Chi_X(m_1, \dots, m_t)=\Chi(X, m_1L_1+ \dots + m_tL_t)$ is a numerical polynomial in $m_1, \dots, m_t$ of degree $\le \dim X$.
Via the identification of the Picard group with $N^1(X)_\Z$ we can thus extend this function to a polynomial function $\Chi_X:N^1(X) \to \R$. We also set $\Chi^T:=\Chi_X\circ T$.
\end{definition}

The following lemma describes the behaviour of $\Chi^T$ with respect to  $W$:

\begin{lemma}\label{lem:transhilb}
For every $\R$-divisor $D$ on $X$ and every reflection $r_i$, the following equality is fulfilled
\begin{equation}
\Chi^T(D)=-\Chi^T(r_i(D))
\label{eq:transhilb}
\end{equation}
\end{lemma}

\begin{proof}
It is enough to notice that the property holds in the lattice $\Pic(X)$ $\subset N^1(X)$, by Lemma \ref{lem:leray}.
\end{proof}

\begin{corollary}\label{cor:transhilb2}
For every $\R$-divisor $D$ on $X$ and every $w \in W$,
\begin{equation}
\Chi^T(D)=\pm\Chi^T(w(D)).
\label{eq:transhilb2}
\end{equation}
\end{corollary}

We will also consider the dual action of $W$ on the dual vector space $N_1(X)$, by considering, for every $w\in W$:
$$
 N_1(X)=\Hom(N^1(X),\R)\,\ni C \to\, w^\vee(C):=C\circ w\in\Hom(N^1(X),\R).
$$
In other words, the action is chosen so that
\begin{equation}\label{eq:dualact}w(D)\cdot C= D\cdot w^\vee(C),\mbox{ for all }D\in N^1(X),\,\,\,C\in N_1(X).\end{equation}
Moreover, the involutions $r_i^\vee$ are reflections as well, given by $$r_i^\vee(C)=C+(K_i\cdot C)\Gamma_i.$$
Note also that the action of $W$ on $N_1(X)$ is clearly faithful, i.e. the morphism $W\to\Gl(N_1(X))$ defined by $w\mapsto w^\vee$ is injective. This observation will help us to prove the following:

\begin{proposition}\label{prop:finiteW}
With the same notation as above, the function $\Chi^T$ vanishes on all the hyperplanes determined by the cycles $w^\vee(\Gamma_i)$, and the group $W$ is finite.
\end{proposition}

\begin{proof}
Since, for every $i$, the reflection $r_i$ fixes the hyperplane $\Gamma_i^\perp$, it follows by Lemma \ref{lem:transhilb} that the polynomial function $\Chi^T$ vanishes on $\Gamma_i^\perp$, for all $i$. But then,  by Corollary \ref{cor:transhilb2}, $\Chi^T$ vanishes on all the hyperplanes of the form $w(\Gamma_i^\perp)$, $w\in W$. In particular, it follows that the cardinality of this set of hyperplanes is smaller than or equal to the degree of $\Chi^T$, i.e. the dimension of $X$.

Every hyperplane $w(\Gamma_i^\perp)$ is uniquely determined by the class of the $1$-cycle $(w^{-1})^\vee(\Gamma_i)$ in the quotient of $N_1(X)$ modulo homotheties, that is, the Grothendieck projectivization $\P(N^1(X))$. Denote by $Z\subset \P(N^1(X))$ the set of elements of the form $w(\Gamma_i^\perp)=[(w^{-1})^\vee(\Gamma_i)]$, $w\in W$ and $i=1,\dots,n$, and consider the diagonal action of $W$ on the (finite) product set $Z^n$.

In order to see that $W$ is finite, it is enough to show that the isotropy subgroup $W^0\subset W$ of elements of $W$ fixing $([\Gamma_1],\dots, [\Gamma_n])$ is finite. By construction, the matrix of every element $w^\vee\in\Gl(N_1(X))$ with respect to the basis $\{\Gamma_1,\dots,\Gamma_n\}$ has integral coefficients and determinant $\pm 1$ (since the properties hold for every generator $r_i^\vee$). On the other hand, if $w\in W^0$, then the matrix of $w^\vee$ with respect to this base must be diagonal, hence its diagonal coefficients are all equal to $\pm 1$. In particular the image of $W^0$ in $\Gl(N_1(X))$ is finite and, since the action of $W$ on $N_1(X)$ is faithful, $W^0$ is finite as well.
\end{proof}

Let us denote by $W^\vee$ the image of $W$ into $\Gl(N_1(X))$. As a consequence of Proposition \ref{prop:finiteW} we may state the following

\begin{corollary}\label{cor:scalar}
With the same notation as above there exists a scalar product $\langle~,~\!\rangle$ in $N_1(X)$ invariant under the action of $W^\vee$. In particular, the reflections $r_i^\vee$ are orthogonal with respect to the scalar product $\langle~,~\!\rangle$, and we may write:
\begin{equation}
-K_i\cdot C=2\dfrac{\langle \Gamma_i,C\rangle}{\langle \Gamma_i,\Gamma_i\rangle}, \mbox{ for all }i = 1, \dots, n.
\label{eq:scalar}
\end{equation}
In particular $\{-K_i,\,i=1,\dots, n\}$ is a basis of $N^1(X)$ as a vector space over $\R$.
\end{corollary}
\begin{proof}
Consider any inner product $(~,~\!)$ in $N_1(X)$. A new inner product $\langle~,~\!\rangle$ defined by $\langle x,y\rangle:=\sum_{w\in W}(w(x),w(y))$, for all $x,y\in N_1(X)$, is $W^\vee$-invariant.
\end{proof}

\begin{remark}\label{rem:inducedscalar}
The inner product $\langle~,~\!\rangle$ provided by the above corollary induces an isomorphism $\varphi:N_1(X)\to N^1(X)$ sending $C$ to $\langle C,-\rangle$. The scalar product in $N^1(X)$ induced by $\langle~,~\!\rangle$ via this isomorphism, that we denote by $\langle~,~\!\rangle$ as well, is invariant by the action of $W$, since
\begin{equation}\label{eq:duality}
w(\varphi(C))=\varphi((w^\vee)^{-1}C)),\mbox{ for all }C\in N_1(X),\mbox{ and }w\in W.
\end{equation}
In particular, the reflections $r_i$ are orthogonal with respect to $\langle~,~\!\rangle$ and we may write:
\begin{equation}\label{eq:scalar2}
-\Gamma_i\cdot D=2\dfrac{\langle D,K_i\rangle}{\langle K_i,K_i\rangle}, \mbox{ for all }i.
\end{equation}
\end{remark}

\subsection{Root systems and Dynkin diagrams}\label{ssec:roots}

We will now show how to use the group $W$ to construct a (reduced) root system in $N^1(X)$; this will allow
us to obtain a finite Dynkin diagram $\cD$ which encodes the intersection matrix $-K_i \cdot \Gamma_j$. In turn this will allow us to associate with $X$ a semisimple Lie group $G$, whose Weyl group is $W$, and
a homogeneous model $G/B$.

\begin{notation}\label{notn:roots} Let us define
$$\Phi:=\left\{w(-K_i)\left|w\in W,\,\,\,i=1,\dots,n\right.\right\}\subset N^1(X),$$
which is a finite set by Proposition \ref{prop:finiteW}.
\end{notation}
The next step will be showing that $\Phi$ is a {\it root system} in $N^1(X)$. This means, by definition (cf. \cite[VI, {\S}1, Definition~1]{B2}), that:
\begin{itemize}
\item[(R1)] $\Phi\subset N^1(X)$ is a finite set of nonzero elements generating $N^1(X)$.
\item[(R2)] For every $D\in\Phi$ there exists $\Gamma_D\in N_1(X)$ such that $D\cdot\Gamma_D=2$ and such that $\Phi$ is invariant by the reflection $r_D$ defined by $r_D(E)=E-(\Gamma_D\cdot E)D$, for all $E\in N^1(X)$.
\item[(R3)] For every $D,D'\in\Phi$, $\Gamma_D\cdot D'\in\Z$.
\end{itemize}
A root system is said to be {\it reduced} if $D,kD\in\Phi$ implies $k=\pm 1$. The {\it Weyl group} of $\Phi$ is the group generated by the reflections $r_D$.

\begin{proposition}\label{prop:root}
With the same notation as above, the set $\Phi$ is a reduced root system in $N^1(X)$, whose Weyl group is $W$.
\end{proposition}

\begin{proof}
Property (R1) follows from Proposition \ref{prop:finiteW} and Corollary \ref{cor:scalar}. For (R2), given $D=w(-K_i)\in\Phi$, we take $\Gamma_D:=(w^{-1})^{\vee}(\Gamma_i)$. Then we have $\Gamma_D\cdot D=-K_i\cdot\Gamma_i=2$ and the reflection $r_D$ leaves $\Phi$ invariant because with this choice of $\Gamma_D$ the corresponding reflection $r_D$ is equal to $w\circ r_i\circ w^{-1}\in W$, as a direct computation shows.
Since, by definition, $\Phi$ is invariant by $W$, this shows that Property (R2) holds and that the $W$ is the Weyl group of $\Phi$.

For the third property, note that the lattice $\Pic(X)$ is invariant by the action of $W$, hence, given $D=w(-K_i)$, $D'=w'(-K_j)$, the element $w^{-1}\circ w'(-K_j)$ is the numerical class of a Cartier divisor, therefore
$$\Gamma_D\cdot D'=(w^{-1})^\vee(\Gamma_i)\cdot w'(-K_j)=(\Gamma_i)\cdot w^{-1}\big(w'(-K_j)\big)\in\Z.$$

Finally, if $\Phi$ was not reduced, there would exist two indices $i$ and $j$, an element $w\in W$ and a real number $k$ with $|k|<1$ such that $w(-K_i)=k(-K_j)$, contradicting the fact that the matrix of every element of $W$ with respect to the basis $\{-K_1,\dots,-K_n\}$ has integral coefficients (every $r_i$ satisfies this property, by definition).
\end{proof}

We will call {\it Cartan matrix} of $X$ the $n\times n$ matrix $\cA$ whose entries are given by $\cA_{ij}:=-K_i\cdot\Gamma_j\in\Z$.
Note that, by definition, $A_{ii}=2$ for all $i$. The following lemma lists the possible values of the non-diagonal entries of $\cA$.

\begin{lemma}\label{lem:obtuse}
With the same notation as above, for every two distinct indices $i,j\in\{1,\dots,n\}$, the possible values of $(-K_i\cdot\Gamma_j,-K_j\cdot\Gamma_i)$ are:
$$
(0,0),\quad (-1,-1),\quad (-1,-2),\quad (-2,-1),\quad (-1,-3),\quad (-3,-1).
$$
\end{lemma}

\begin{proof}
Note first that $\cA_{ij}=2{\langle -K_i,-K_j\rangle}/{\langle -K_j,-K_j\rangle}$, by equation (\ref{eq:scalar2}). Then, being $\Phi$ a root system, the possible values of $\cA_{ij}$ are listed in \cite[9.4]{H}. It is then enough to prove that $\cA_{ij}\leq 0$ for $i\neq j$. In order to see this we consider the fiber product $S$ of $\pi_i:X\to X_i$ and its restriction to a curve $\Gamma_j$:
$$
\xymatrix{S\ar[r]\ar[d]&X\ar[d]^{\pi_i}\\\Gamma_j\ar@/^/[u]^{s_j} \ar[ru]\ar[r]_{\pi_i|_{\Gamma_j}}&X_i}
$$

The inclusion of  $\Gamma_j$ into $X$ factors via $S$, providing a morphism: $s_j:\Gamma_j\to S$, whose image satisfies: $-K_{S|\Gamma_j}\cdot s_j(\Gamma_j)=-K_i\cdot\Gamma_j$. Now, if this intersection number were positive, one could easily prove that the curve $s_j(\Gamma_j)$ would be algebraic equivalent to a reducible $1$-cycle. From this we would obtain that the family of $\Gamma_j$'s in $X$ would not be unsplit, a contradiction.
\end{proof}

\begin{notation}\label{notn:dynkin} The matrix $\cA$ is determined by its {\it Dynkin diagram} $\cD$, which consists of a graph with $n$ ordered nodes, where the nodes $i$ and $j$ are joined by $\cA_{ij}\cA_{ji}$ edges. When two nodes $i$ and $j$ are joined by a double or triple edge, we add to it an arrow, pointing towards $i$ if $\cA_{ij}>\cA_{ji}$.
\end{notation}

\begin{proposition}
\label{prop:classifdynkin}
With the same notation as above, every connected component of the Dynkin diagram $\cD$ is one of the following. $${\rm A}_n,\,\,{\rm B}_n,\,\,{\rm C}_n,\,\,{\rm D}_n,\,\,{\rm E}_6,\,\,{\rm E}_7,\,\,{\rm E}_8,\,\,{\rm F}_4,\mbox{ or~} {\rm G}_2.$$
\end{proposition}

\begin{proof}
Without loss of generality, we may assume that $\cD$ is connected.
Note that, by Lemma \ref{lem:obtuse} above, in the euclidean space $(N^1(X),\langle\,,\,\rangle)$, the set $$\{\epsilon_i:=-K_i/\langle -K_i,-K_i\rangle^{\frac{1}{2}}|\,\,i=1,\dots,n\}$$ is {\it admissible}, that is, its elements are $n$ linearly independent unit vectors satisfying $\langle \epsilon_i,\epsilon_j\rangle\leq 0$ for $i\neq j$, and $4\langle \epsilon_i,\epsilon_j\rangle^2=0,1,2$ or $3$ for $i\neq j$. Then the result follows from the proof of  \cite[Theorem~11.4]{H}.
\end{proof}

In particular, we may associate with $X$ a semisimple Lie group $G$, and its semisimple Lie algebra $\fg$, determined by the Dynkin diagram $\cD$. We choose a Cartan subgroup $H$, with Lie subalgebra $\fh\subset\fg$ and denote by $\overline{\Phi}\subset\fh^\vee$ the corresponding root system (cf. \cite[8]{H}). Moreover, we denote by $\fh^\vee_\R$ the real vector subspace of $\fh^\vee$ generated by $\overline{\Phi}$. Given a base of simple roots $\overline{\Delta}=\{ \alpha_i|\,\,i=1,\dots,n\}$  (see \cite[10.1]{H} for the definition), we may assume that its elements have been ordered so that  the corresponding Dynkin diagram is equal to $\cD$. We will consider the morphism
$$
\psi:N^1(X)\to \fh^\vee_\R,\mbox{ defined by }\psi(-K_i)=\alpha_i.
$$

\begin{definition}\label{notn:model}
Consider the Borel subgroup $B\subset G$ containing the Cartan subgroup $H$. The complete flag manifold $G/B$ is a rational homogeneous space, which we will call the {\it rational homogeneous model of} $X$. We will use for $G/B$ a similar notation as for $X$, adding an overline to distinguish the two cases (so we will use $\overline{\pi_i}, \overline{\Gamma_i},-\overline{K_i},\dots$).
\end{definition}

\begin{corollary}\label{cor:simpleroots}
With the same notation as above, the set $\Delta:=\{-K_i|\,\,i=1,\dots,n\}$ is a base of simple roots of $\Phi$.
\end{corollary}

\begin{proof}
The Dynkin diagram $\cD$ contains the necessary information to recover $\overline{\Phi}$ up on $\overline{\Delta}$ and the corresponding reflections. Then, by construction, the isomorphism $\psi$ defined above maps $\overline{\Delta}$ to $\Delta$ and $\overline{\Phi}$ to $\Phi$, so the claim follows.
\end{proof}

\begin{corollary}\label{cor:chambers}
With the same notation as above, $\Amp(X)$ is a Weyl chamber of $W$, and $\Nef(X)$ is a fundamental domain of $W$.
\end{corollary}

\begin{proof}
The fact that $\Delta$ is a base of simple roots is equivalent to the set $$U=\{D\in N^1(X)|\,\,\langle D,-K_i\rangle> 0,\,\,i=1,\dots,n\}$$ to be a Weyl chamber of $\Phi$ (see \cite[10.1]{H} for details). By equation (\ref{eq:scalar2}), $U$ is the set of classes having positive intersection with every $\Gamma_i$. Hence, by Lemma \ref{lem:basis}, $U$ is the ample cone of $X$. The second part of the statement follows then from \cite[10.2]{H}.
\end{proof}

\begin{remark}\label{rem:positive}
Recall that, being $\Delta$ a base of simple roots, the root system $\Phi$ decomposes as a union $\Phi=\Phi^+\cup\Phi^-$, where $\Phi^+$ (respectively $\Phi^-$) is the set of roots that are integer linear combination of elements of $\Delta$ with nonnegative (resp. nonpositive) coefficients. The elements in $\Phi^+$ and $\Phi^-$ are called, respectively, positive and negative roots of $\Phi$ with respect to $\Delta$. Moreover, by Corollary \ref{cor:chambers}, taking any ample divisor $A$, the set  $\Phi^+$ (resp. $\Phi^-$) can be described as the set of roots $D$ satisfying $\langle A,D\rangle > 0$ (resp. $\langle A,D\rangle < 0$). See \cite[10.1]{H} for details.
\end{remark}

\subsection{Length of divisors and cohomology}\label{ssec:BWB}

We will  now show that the cohomology of divisors on $X$ equals the one on the homogeneous model, i.e.
$h^i(X,D)=h^i(G/B,\psi(D))\mbox{ for every }D\in\Pic(X),\,\,i\in\Z$; in particular
the dimension of $X$ is equal to the dimension of its homogeneous model.

We will call {\it translated Weyl chambers}  the images by $T$ of the Weyl chambers of $\Phi$. In other words, they are the connected components of $N^1(X)\setminus\bigcup_{D\in\Phi}T(D^\perp)$.

\begin{lemma}\label{lem:W'}
With the same notation as above, for every $D\in N^1(X)$, there exist $w'\in W'$  and a unique $D'\in T(\Nef(X))$ such that $w'(D)=D'$.
If moreover $D$ lies in a translated Weyl chamber of $X$, $w'$ is uniquely determined by $D$.
\end{lemma}

\begin{proof}
Clearly, $W'$ acts on the set of translated Weyl chambers of $\Phi$. Then the result follows from Corollary \ref{cor:chambers} and \cite[10.3]{H}.
\end{proof}

Following \cite[10.3]{H}, given an element $w\in W$, we will define its {\it length} $\len(w)$ as the minimal $t$ such that $w$ can be written as $w=r_{i_1} \circ \dots \circ r_{i_t}$; such an expression for $w$ will be called {\em reduced}; the length of the identity is $0$, by definition. Then, for every element $w'\in W'$, we define its {\it length} as $\len(w'):=\len(T^{-1}\circ w'\circ T)$. In view of the previous lemma, for an element $D$ lying in a translated Weyl chamber, we define its {\it length} $\len(D)$ as the length of the  element $w'$ sending it into $T(\Amp(X))$. The following result relates the length $\len(D)$ with the cohomology of $D$:

\begin{proposition}\label{prop:BBWFT1}
With the same notation as above, given  $D\in\Pic(X)$, we consider an element $w'\in W'$ such that $w'(D)\in T(\Nef(X))$. Then:
\begin{itemize}
\item either $w'(D)\in T(\Nef(X)\setminus\Amp(X))$ and $H^j(X,D)=0$ for all $j$,
\item or $w'(D)\in T(\Amp(X))$ %and $H^{s(D)}(X,D)\neq 0$
and $H^j(X,D)=0$ for $j\neq \len(D)$.
\end{itemize}
\end{proposition}

\begin{proof}
We already know that, for any $D\in\Pic(X)$ there exists an index $i_0$ for which $H^j(X,D)=0$ for $j\neq i_0$ (see Lemma \ref{lem:cohomW}). If $D\in w'(T(\Nef(X)\setminus \Amp(X)))$ for some $w'\in W'$ then, there exists an index $i$ such that $T^{-1}(w'^{-1}(D))\cdot \Gamma_i=0$. Hence, if $w=T^{-1}\circ w'\circ T$, we may assert that $T^{-1}(D)\cdot (w^{-1})^{\vee}(\Gamma_i)=0$ and by Proposition \ref{prop:finiteW}, $\Chi(D)=\Chi^T(T^{-1}(D))=0$. From this we conclude that $H^j(X,D)=0$ also for $j=i_0$.

If $D\in w'(T(\Amp(X))$, then $w'$ is uniquely determined by $D$. Write $w'=r'_{i_1}\circ\dots\circ r'_{i_t}$, with $t=\len(D)$, and denote by $A$ the only ample class such that $D=w'(T(A))$. Following the idea of the proof of \cite[10.3, Lemma A]{H}, we will prove our statement by induction on $t$. The statement is clear for $t=0$, and we will assume by induction that $H^j(X,r'_{i_2}\circ\dots\circ r'_{i_t}(T(A)))=0$ for $j\neq t-1$. Then, by Lemma \ref{lem:leray}, it suffices to show that $r'_{i_2}\circ\dots\circ r'_{i_t}(T(A))\cdot\Gamma_{i_1} \geq -1$. Assume by contradiction that $r_{i_2}\circ\dots\circ r_{i_t}(A)\cdot\Gamma_{i_1}< 0$.

Setting $w_1:=r_{i_2}\circ\dots\circ r_{i_t}$, this inequality tells us that $\langle w_1(A),-K_{i_1}\rangle< 0$, that is $\langle A,w_{1}^{-1}(-K_{i_1})\rangle< 0$. This implies (see Remark \ref{rem:positive}) that $w_{1}^{-1}(-K_{i_1})$ is a negative root, which, by \cite[10.2 Lemma C]{H}, tells us that
$$
w_{1}^{-1}\circ r_{i_1}=r_{i_t}\circ\dots\circ r_{i_2}\circ r_{i_1}=r_{i_t}\circ\dots\circ r_{i_{s+1}}\circ r_{i_{s-1}}\circ\dots\circ r_{i_2},%\mbox{ for some }s<t.
$$
for some $s<t$. In particular this implies that $\len(r_{i_t}\circ\dots\circ r_{i_2}\circ r_{i_1})<t$, a contradiction.
\end{proof}

Let $G/B$ be the homogenous model of $X$. Every line bundle on $G/B$ is homogeneous, hence determined by a {\it weight}, that is, an element of $\fh^\vee$. In particular, $-\overline{K_i}$ is the homogeneous line bundle on $G/B$ associated to the simple root $\alpha_i$, and $N^1(G/B)$ may be identified with real vector space $\fh^\vee_{\R}$. In the sequel, we will interpret the morphism $\psi$ defined above as an isomorphism from $N^1(X)$ to $N^1(G/B)$ sending $-K_i$ to $-\overline{K_i}$ for every $i$. Moreover, this isomorphism sends $-K_X$ to $-K_{G/B}$. In fact, since the Cartan matrix $\cA$ is nonsingular, the coefficients of $-K_X$ with respect to the basis $\{-K_1,\dots,-K_n\}$ are determined by the intersection numbers $-K_X\cdot\Gamma_i=2$.
The same holds for $G/B$ so, in particular, denoting by $\overline{T}$ the translation by $\frac{1}{2}K_{G/B}$ in $N^1(G/B)$, we have $\psi \circ T = \overline T \circ \psi$.
Note finally that, by construction, $\psi$ preserves the length of divisors. As a consequence, we get:

\begin{proposition}\label{prop:dim}
With the same notation as above, let $G/B$ be the homogenous model of $X$. Then $\dim X =\dim G/B$.
\end{proposition}

\begin{proof}
Since, by Serre duality, $H^{\dim X}(X,K_X)\neq 0$, the dimension of $X$ equals, by Proposition \ref{prop:BBWFT1}, the length of  $K_X$. But, by our previous observations on the map $\psi$, this equals the length of $K_{G/B}$, which, by the same argument, is equal to $\dim(G/B)$.\end{proof}

For every $D\in\Phi$, denote by $F_D:N^1(X)\to\R$ the linear operator defined by:
$$
F_D(L):=-2\dfrac{\langle D,L\rangle}{\langle D,K_X\rangle}.
$$
The following result states that $\Chi_X$ depends only on the root system $\Phi$.

\begin{theorem}\label{thm:chi}
Let  $\Phi^+ \subset \Phi$ be the set of positive roots with respect to $\Delta$. Then
$$\Chi_X=\prod_{D\in\Phi^+}\left(1+F_D\right).$$
\end{theorem}

\begin{proof} We will show that $\Chi^T=\Pi_{D\in\Phi^+}F_D$.
By Proposition \ref{prop:dim} $\deg(\Chi^T)=\deg(\Chi)\le \dim X=\dim G/B$. Moreover, since  $B$ is the subgroup associated to a Borel subalgebra $\fb$ which is   isomorphic to the direct sum of the root spaces $\fg_\alpha$, $\alpha\in \overline\Phi\,\!^+$ (cf. \cite[II.8]{H}), it follows that $\dim G/B=\sharp(\overline{\Phi}\,\!^+)$.
Since $\Phi$ is a reduced root system, two different elements in $\Phi^+$ are not proportional, hence
$\Pi_{D\in\Phi^+}F_D$ vanishes on the $|\Phi^+|$ distinct hyperplanes $\{F_D=0\}$.
Since   $\Chi^T(-K_X/2)=1$ and $F_D(-K_X/2)=1$ for all $D\in\Phi^+$, it is enough to show that $\Chi^T$ vanishes on the hyperplane $\{F_D=0\}$ for every $D\in \Phi^+$.
Given any  $D \in \Phi^+$ there exists $w \in W$ such that $D=w(-K_i)$ for some $i$, and the hyperplane
$\{F_D=0\}$ is $\{D' \in N_1(X) \,|\, D' \cdot w^\vee(\Gamma_i) = 0\}=w(\Gamma_i^\perp)$, with $\Gamma_i^\perp$ as in (\ref{eq:hyper}); we already noticed  in the proof of Proposition \ref{prop:finiteW} that $\Chi^T$ vanishes on these hyperplanes.
\end{proof}

As a consequence of Proposition \ref{prop:BBWFT1} and Theorem \ref{thm:chi}, we clearly obtain the following  result:

\begin{corollary}\label{cor:cohomequal}
With the same notation as above, for every line bundle belonging to the subgroup of $Pic(X)$ generated by $K_1, \dots, K_n$
 $$h^i(X,L)=h^i(G/B,\psi(L))\quad i\in\Z.$$
\end{corollary}

\section{Bott-Samelson varieties %of weak FT-manifolds
}\label{sec:BSvar}

Let $I=\{1,2,\dots,n\}$ be the set of indices parametrizing the nodes of the Dynkin diagram $\cD$ of $X$, and let $G/B$ be the homogeneous model of $X$.
Following \cite{LT}, when dealing with finite sequences of indices in $I$ we will use this notation:

\begin{notation}\label{notn:LT} Given a sequence $\ell=(l_1,\dots,l_r)$, $l_i\in I$, we will set
$$\left.\begin{array}{l}\ell[s]:=(l_1,\dots,l_{r-s})\\
\ell(i):=(l_1,\dots,\widehat{\,l_i},\dots,l_r)\\
\ell[s](i):=(l_1,\dots,\widehat{\,l_i},\dots,l_{r-s})
\end{array}\right\}\mbox{ for }i \leq
r-s\leq r $$
In particular $\ell[s][s']=\ell[s+s']$ and $\ell[s](r-s)=\ell[s+1]$.
If $\ell=(l_1, \dots l_r)$ and $\ell'=({l}_1', \dots, l'_t)$ are two sequences as above we will denote
by $\ell\ell'$ the sequence $(l_1, \dots l_r, l'_1, \dots, l'_t)$. Moreover we will denote by $\ell^k$ the sequence obtained juxtaposing $k$ times the indices of $\ell$.
Finally
we will denote by $w(\ell)$ the element of the Weil group $W$ of $\cD$ given by $w(\ell):=r_{l_1}\circ r_{l_2}\circ\dots\circ r_{l_r}$.
\end{notation}

Fix a point $x \in X$. To every sequence $\ell=(l_1,\dots,l_r)$ in $I$ we will associate a sequence of manifolds $Z_{\ell[s]}$, $s=0,\dots,r$, called the {\it Bott-Samelson varieties} associated to the subsequences $\ell[s]$, together with morphisms
$$f_{\ell[s]}:Z_{\ell[s]} \to X,\quad p_{\ell[s+1]}:Z_{\ell[s]}\to Z_{\ell[s+1]},\quad \sigma_{\ell[s+1]}:Z_{\ell[s+1]}\to Z_{\ell[s]}.$$
They are constructed in the following way: for $s=r$ we set $Z_{\ell[r]}:=\{x\}$ and let $f_{\ell[r]}:\{x\}\to X$ be the inclusion. Then for $s<r$ we define $Z_{\ell[s]}$ recursively upon $Z_{\ell[s+1]}$
 by considering the composition
$g_{\ell[s+1]}:=\pi_{l_{r-s}}\circ f_{\ell[s+1]}:Z_{\ell[s+1]}\to X_{l_{r-s}}$ and taking its fiber product with $\pi_{l_{r-s}}$:
$$
\xymatrix@=35pt{Z_{\ell[s]}\ar[r]^{f_{\ell[s]}}\ar[d]^{p_{\ell[s+1]}}&X\ar[d]^{\pi_{l_{r-s}}} \\
Z_{\ell[s+1]}\ar@/^/[u]^{\sigma_{\ell[s+1]}} \ar[ru]_{f_{\ell[s+1]}}\ar[r]^{g_{\ell[s+1]}}&X_{l_{r-s}}}
$$
Note that the map $f_{\ell[s+1]}$ factors naturally via $Z_{\ell[s]}$, providing a section $\sigma_{\ell[s+1]}:Z_{\ell[s+1]}\to Z_{\ell[s]}$ of $p_{\ell[s+1]}$. This shows in particular that $p_{\ell[s+1]}$ is a $\P^1$-bundle. Pulling back to $Z_{\ell[s+1]}$ via $\sigma_{\ell[s+1]}$
the relative Euler sequence over $Z_{\ell[s]}$ we see that this $\P^1$-bundle is
 the projectivization of an extension $\cF_{\ell[s]}$ of $\cO_{Z_{\ell[s+1]}}$ by $f_{\ell[s+1]}^*K_{l_{r-s}}$:
\begin{equation}\label{eq:F}
0 \lra f_{\ell[s+1]}^*K_{l_{r-s}}\longrightarrow
\cF_{\ell[s]} \longrightarrow \cO_{Z_{\ell[s+1]}}\lra 0.
\end{equation}
The cohomology class determining it will be denoted by
\begin{equation}
\zeta_{\ell[s]}\in H^1(Z_{\ell[s+1]},f_{\ell[s+1]}^*K_{l_{r-s}}).\label{eq:class}
\end{equation}

The next lemma will be used to prove that if $X$ has disconnected Dynkin diagram, then it is a product (see Corollary \ref{cor:product}).
\begin{lemma}\label{lem:prodBS}
Assume that $\cD = \cD_1 \sqcup \cD_2$, and denote by $I_1, I_2 \subset I$ the sets of indices parametrizing nodes in $\cD_1$ and $\cD_2$, respectively. Let $\ell_1=(l_1,\dots, l_k), \ell_2=(l'_1, \dots, l'_s)$ be sequences of indices in $I_1$ and $I_2$, respectively. Then $Z_{\ell_1\ell_2} = Z_{\ell_1} \times Z_{\ell_2}$.
\end{lemma}

\begin{proof} We will use induction on the number of elements of $\ell_2$.
If $\ell_2$ consists only of the element $l'_1$ then $Z_{\ell_1\ell_2}$ is obtained as the projectivization of the rank two vector bundle
\begin{equation}
0\lra f_{\ell_1}^*K_{l'_1}\longrightarrow
\cF_{\ell_1} \longrightarrow \cO_{Z_{\ell_1}}\lra  0.
\end{equation}
Since $K_{l_{1}'}$ is trivial on $Z_{\ell_1}$ also $\cF_{\ell_1}$ is trivial.\\
For the general case, let us consider the following commutative diagram

$$\xymatrix@=35pt{
Z_{\ell_1\ell_2} \ar@{}[dr]|{\mbox{(1)}}     \ar[r] \ar[d]_{p_{\ell_1\ell_2[1]}}  & Z_{\ell_2}  \ar@{}[dr]|{\mbox{(2)}}  \ar[d]_{p_{\ell_2[1]}}  \ar[r]^{f_{\ell_2}}  &X \ar[d]^{\pi_{l_s'}}   \\
Z_{\ell_1\ell_2[1]} \ar@{}[dr]|{\mbox{(3)}}   \ar[d]_{} \ar[r]  &  Z_{\ell_2[1]} \ar[r]_{g_{\ell_2[1]}}\ar[d]_{} & X_{l_s'} \\
Z_{\ell_1} \ar[r] &
\Spec(\C)
& \\
}$$

Squares (2)  and (1)--(2) are Cartesian, by construction of the Bott-Samelson varieties $Z_{\ell_2}$ and $Z_{\ell_1\ell_2}$, hence also square (1) is Cartesian.
Square (3) is Cartesian by the induction assumption, hence also square (1)--(3) is Cartesian.
\end{proof}

\subsection{Divisors and $1$-cycles on Bott-Samelson varieties}\label{ssec:divBS}
Let $Z_\ell$ be the Bott-Samelson variety corresponding to $\ell=(l_1,\dots, l_r)$. In this section we will describe two pairs of dual basis for the vector spaces $N^1(Z_\ell)$ and $N_1(Z_\ell)$.
\begin{notation}
Note that, by construction, linear and numerical equivalence of divisors on $Z_\ell$ are the same.
In particular, as we did for $X$,
we will abuse of notation, denoting the numerical class of a divisor and the corresponding line bundle with the same letter, using additive notation.
\end{notation}

Let us denote by $\beta_{i(r-i)}$ the class in $N_1(Z_{\ell[r-i]})$ of the fibers of $p_{\ell[r-i+1]}:Z_{\ell[r-i]} \to Z_{\ell[r-i+1]}$. We  will denote by $\beta_{i(s)}$ the image of this class into $N_1(Z_{\ell[s]})$, via push forward with the sections $\sigma_{\ell[r-j]}$, $j= i, \dots, r-s-1$. If $s=0$ we will write $\beta_{i}$ instead of $\beta_{i(0)}$. Note that, by construction, $f_{\ell[s]*}\beta_{i(s)}=[\Gamma_{l_i}]$. The next remark describes the behaviour of the cohomology class $\zeta_{\ell[s]}$ with respect to curves in the classes $\beta_{i(s+1)}$:

\begin{remark}\label{rem:sections}
Consider the projection $p_{\ell[s+1]}:Z_{\ell[s]}\to Z_{\ell[s+1]}$, whose fibers lie in the class $\beta_{r-s(s)}$, for some $s\geq 0$. Assume that there exists $i<r-s$ satisfying that $l_i=l_{r-s}$. Then, over a rational curve $C\cong\P^1$ in $Z_{\ell[s+1]}$ whose class is $\beta_{i(s+1)}$ the sequence (\ref{eq:F}) restricts to a sequence of the form
$$0\lra  \cO_{\P^1}(-2)\longrightarrow \cO_{\P^1}(-1) \oplus \cO_{\P^1}(-1)
\longrightarrow \cO_{\P^1} \lra  0$$
and the section ${\sigma_{\ell[s+1]}}$ restricted to $C$ is the diagonal morphism. In particular $\beta_{i(s)}={\sigma_{\ell[s+1]}}_*\beta_{i(s+1)}= \beta_{r-s(s)}+{\beta'\!\!}_{i(s)}$, where ${\beta'\!\!}_{i(s)}$ denotes the class of a minimal section of $p_{\ell[s+1]}$ over $C$.

If else $i < r-s$ is such that $l_i \not =l_{r-s}$ then the restriction of the sequence (\ref{eq:F}) to a curve $C\cong\P^1$ in the class $\beta_{i(s+1)}$ splits (by Lemma \ref{lem:obtuse}), and $\beta_{i(s)}={\sigma_{\ell[s+1]}}_*\beta_{i(s+1)}$ is the class of a minimal section $p_{\ell[s+1]}$ over $C$.
\qed
\end{remark}

\begin{corollary}\label{cor:nonsplit} Let $\ell=(l_1, \dots, l_r)$ be a sequence in $I$. Assume that there exists $i<r-s$ satisfying that $l_i=l_{r-s}$. Then the class $\zeta_{\ell[s]}$ defined in (\ref{eq:class}) is not trivial. In particular
$h^1(Z_{\ell[s+1]},f_{\ell[s+1]}^*K_{l_{r-s}}) \not = 0$.

\end{corollary}

For every $i\leq r-s$ we will denote by $Z_{\ell[s](i)}$ the divisor in $Z_{\ell[s]}$ which is  the image of the natural inclusion of Bott-Samelson varieties $Z_{\ell[s](i)}\hookrightarrow Z_{\ell[s]}$. It may also be described as the pull-back to $Z_{\ell[s]}$, via the projections $p_{\ell[r-j]}$, $j=i , \dots, r-s-1 $ of the  divisor $\sigma_{\ell[r-i+1]}(Z_{\ell[r-i+1]})\subset Z_{\ell[r-i]}$.
As usual, for $s=0$ we write $Z_{\ell(i)}$ instead of $Z_{\ell[0](i)}$.
Abusing notation we will denote by $Z_{\ell[s](i)}$ also the numerical class of $Z_{\ell[s](i)}$ in $N^1(Z_\ell)$. The following lemma is immediate:

\begin{lemma}\label{lem:baseBS} (Cf. \cite[Section 2 and Subsection 3.2]{LT}) The divisor $Z= \sum Z_{\ell(i)}$ is  simple normal crossing.
Moreover the $Z_{\ell(i)}$'s (resp. the $\beta_{i}$'s), $i=1,\dots, r$, form a basis of $N^1(Z_{\ell})$ (resp. of $N_1(Z_{\ell})$).
\end{lemma}

Within $N^1(Z_\ell)$ we may also consider the dual basis of $\{\beta_i,\,i=1,\dots,r\}$, that we denote by $$\{H_i,\,i=1,\dots,r\}.$$ By our previous description of the Bott-Samelson varieties as $\P^1$-bundles -- equation (\ref{eq:F}) --, every $H_i$ is  the  tautological bundle of $Z_{\ell[r-i]}\cong\P(\cF_{\ell[r-i]})$, that is, ${p_{\ell[r-i]}}_*H_i=\cF_{\ell[r-i]}$. In particular, these line bundles provide a $\Z$-basis of $\Pic(Z_{\ell})$.

The following lemma provides an adjunction formula for the evaluation morphisms $f_\ell:Z_\ell\to X$.

\begin{lemma}\label{lem:canonical}(Cf. \cite[Lemma~5.1]{LT})
With the same notation as above:
 \begin{equation}
f_\ell^*(-K_X/2)=\sum_{i=1}^rH_i\quad \mbox{and} \quad -K_{Z_\ell}= \sum_{i=1}^rZ_{\ell(i)} + f^*_\ell(-K_X/2).\label{eq:canonical}
\end{equation}
\end{lemma}

\begin{proof}
Since $f^*_\ell(-K_X/2)\cdot \beta_i=1$ for all $i$, it follows that $f^*_\ell(-K_X/2)=\sum_{i=1}^rH_i$. On the other hand,
note that, for every $s\leq r$ one has $$f_{\ell[s+1]}^*(-K_{l_{r-s}})=\big(Z_{\ell[s](r-s)}\big)_{|Z_{\ell[s+1]}}.$$
Then using adjunction formula recursively, one gets:
$$
-K_{Z_\ell}=\sum_{i=1}^rZ_{\ell(i)}+\sum_{i=1}^rH_i = \sum_{i=1}^rZ_{\ell(i)} + f^*_\ell(-K_X/2).
$$%\par\vspace{-0.85 cm}
\end{proof}
%\medskip

Let us also define for every $t\leq r$, the following line bundles on $Z_\ell$:
\begin{equation}
N_{t}=\sum_{\substack{i\leq t \,\,\,l_i=l_t}}H_{i}.\label{eq:Ns}
\end{equation}
and the classes $\gamma_i\in N_1(Z_{\ell})$, $i=1,\dots,r$, defined by $N_{t}\cdot\gamma_{i}=\delta_{i}^{t}$.

Since, by definition, $N_{t}\cdot\beta_{i}=1$ if $l_t=l_i, i\leq t$,  and it is zero otherwise, the intersection matrix $(N_t\cdot\beta_i)$ is lower triangular and its diagonal entries are ones.

The classes $\gamma_i$ might also be defined as follows:
for every index $i\leq r$, satisfying that the set $\{k\in\Z|\,\,i<k\leq r,\,\,l_i=l_k\}\neq\emptyset$ denote by $i^*$ the minimum of this set. Then, according to Remark  \ref{rem:sections}, the class
$\beta_{i(r-i^*)}$ decomposes in $N_1(Z_{\ell[r-i^*]})$ as $\beta_{i^*(r-i^*)}+{\beta'\!\!}_{i(r-i^*)}$. Defining $\beta'_{i}$ as the image of ${\beta'\!\!}_{i(r-i^*)}$ into $N_1(Z_{\ell})$, via push forward with the sections $\sigma_{\ell[r-j]}$, $j= i^*, \dots, r-1$, we may then write:
$$
\gamma_{i}=\left\{\begin{array}{ll}\beta_{i}&\mbox{if }\{k\in\Z|\,\,i<k\leq r,\,\,l_i=l_k\}=\emptyset\\	\beta'_{i}&\mbox{otherwise}\end{array}\right.
$$

\subsection{The Mori cone of the Bott-Samelson varieties}
In this section we will show that the Nef and Mori cones of the Bott-Samelson varieties $Z_\ell$ are simplicial, describing their generators. Moreover we will describe the relative cone of the maps $f_\ell:Z_\ell \to X$.

\begin{lemma}\label{lem:O(1)}
The divisor $N_{t}$ is nef on $Z_\ell$, for every $t\leq r$.
\end{lemma}

\begin{proof}
Set $k=l_t$ and consider the contraction $\pi^k:X\to X^k$ associated with the $(n-1)$-dimensional face of $\cNE{X}$ generated by the rays $R_i$, $i\neq k$, whose image $X^k$ has Picard number one. Denote by $p:Z_\ell \to Z_{\ell[r-t]}$ the composition of the projections $p_{\ell[1]}, \dots, p_{\ell[r-t]}$.

By construction, $\varphi_t:=\pi^k\circ f_{\ell[r-t]}\circ p$ contracts every curve whose class is $\beta_{j}$ such that  $l_{j}\neq k$ and every curve whose class is $\beta_{j}$ with $l_j=k$ and $j >t$. Moreover, it contracts every curve whose class is ${\beta'\!\!}_{j}$, for any $j\neq t$, since either these curves are contracted by $p$ if $j>t$, or they are contracted by $f_{\ell[r-t]}\circ p$ if $j<t$. On the other hand, the image of $\beta_{t}$ is $\pi^k(\Gamma_k)$, which is not a point.

We conclude that the kernel of induced linear map ${\varphi_t}_*:N_1(Z_\ell)\to N_1(X^k)$ contains the linear space generated by the classes $\gamma_{i}$, $i\neq t$, which is a hyperplane. This hyperplane meets $\overline{\NE}(Z_\ell)$ along the extremal face defined by $\varphi_t$. Since $N_{t}$ vanishes on this hyperplane and it has positive intersection with $\gamma_{t}$, it follows that it is a supporting divisor of $\varphi_t$ and, in particular, it is nef.
\end{proof}

\begin{corollary}\label{cor:coneBS}
The Mori cone (respectively, the nef cone) of $Z_\ell$ is the simplicial cone generated by the classes $\gamma_{t}$ (resp.  $N_{t}$), $t\leq r$.
\end{corollary}

\begin{proof} Let $\cC$ (respectively, $\cK$) be the polyhedral cone generated by the $\gamma_{t}$'s (resp. by the $N_{t}$'s). Note that, by definition of the $\gamma_{t}$'s, these two cones are dual. Moreover, the cone $\cC$ is contained in $\overline{\NE}(Z_\ell)$ and,
by Lemma \ref{lem:O(1)}, $\cK$ is contained in $\Nef(Z_\ell)$. Summing up we have:
$$
\cK\subseteq\Nef(Z_\ell)=\big(\overline{\NE}(Z_\ell)\big)^\vee\subseteq\cC^\vee=\cK.
$$
This concludes the proof.
\end{proof}

\begin{corollary}\label{cor:stein}
Set $J=\{i\,\, |\,\,l_i=l_k\mbox{ for some } k>i\}$; then the Stein factorization of the map $f_\ell:Z_\ell\to X$ is the contraction of $Z_\ell$ associated to the extremal face of $\cNE{Z_\ell}$ generated by $\{\gamma_i\,\,|\,\, i  \in J\}$.
\end{corollary}

\begin{proof}
By Lemma \ref{lem:canonical} and equation (\ref{eq:Ns}) $f_\ell^*(-K_X/2)= \sum_{i \not \in J} N_i$, and this nef divisor vanishes on the face spanned by  $\{\gamma_i\,\,|\,\, i  \in J\}$.
\end{proof}

We will denote by $\overline{Z}_\ell$, $\overline{f}_\ell$, etc., the Bott-Samelson varieties and the corresponding morphisms, associated with the homogeneous model $G/B$, for every list $\ell$. We will consider also the linear isomorphism $$\psi_\ell:N^1({Z}_\ell)\to N^1(\overline{Z}_\ell)$$ defined by  $\psi_\ell({Z}_{\ell(i)})=\overline{Z}_{\ell(i)}$ for all $i$.
By construction, its transposed map $\psi_\ell^t:N_1(\overline{Z}_\ell)\to N_1({Z}_\ell)$, defined by the equality $\psi_\ell^t(C)\cdot D=C\cdot\psi_\ell(D)$ for all $C\in N_1(\overline{Z}_\ell)$, $D\in N^1({Z}_\ell)$, satisfies $\psi_\ell^t(\overline{\beta}_i)=\beta_i$, hence we have
the following

\begin{corollary}\label{cor:psi}
With the same notation as above, for every list $\ell$, the map $\psi_\ell$ sends $\Nef({Z}_\ell)$ to $\Nef(\overline{Z}_\ell)$, $K_{{Z}_\ell}$ to $K_{\overline{Z}_\ell}$, and $f^*_\ell(K_X)$ to $\overline{f}^*_\ell(K_{G/B})$.
\end{corollary}

\begin{proof}
The first part follows from Corollary \ref{cor:coneBS}. The second follows from the fact that the numerical classes of $K_{{Z}_\ell}$ and  $f^*_\ell(K_X)$ are determined by their intersections with the classes $\beta_i$, and these intersection numbers are the same for $X$ and its  homogeneous model.
\end{proof}

\subsection{Cohomology of line bundles on Bott-Samelson varieties}

In this section we will describe some properties of cohomology of line bundles on Bott-Samelson varieties, that will be the key tool to prove the uniqueness of these varieties in Section \ref{sec:unique}. Moreover, we will use them here to prove that the evaluation morphisms of Bott-Samelson varieties corresponding to a reduced expression of the longest element in $W$ are birational.

\begin{remark}\label{rem:descent}
Note that, for any sequence $\ell=(l_1,\dots,l_r)$ and every line bundle $L$ on $X$, setting $s=
%f_{\ell}^*(L)\cdot \beta_{r}=
L\cdot\Gamma_{l_r}$, the following equalities are obvious:
\begin{equation}\label{eq:proj1}
f_{\ell}^*L=sH_r+p_{\ell[1]}^*f_{\ell[1]}^*L,
\end{equation}
\begin{equation}\label{eq:proj2}
{p_{\ell[1]}}_*f_{\ell}^*L={p_{\ell[1]}}_*sH_r \otimes f_{\ell[1]}^*L,\end{equation}
\end{remark}

\begin{lemma}\label{lem:j2} Let $Z_\ell$ be the Bott-Samelson variety of $X$ defined by the sequence
$\ell=(l_1,\dots,l_r)$, and let $L$ be a line bundle on $X$ of degree $s$ with respect to $\Gamma_{l_r}$. Then:

\begin{enumerate}
\item If $s\geq 0$
then $H^i(Z_\ell, f^*_\ell L)=H^i(Z_{\ell[1]},{p_{\ell[1]}}_*f^*_\ell L)$ for every $i$ and the rank $s+1$ bundle ${p_{\ell[1]}}_*f^*_\ell L$ has a filtration on $Z_{\ell[1]}$ with factors
$$f_{\ell[1]}^*\big(L+ tK_{l_r}\big),\,\,\,t=0,\dots,s.$$
\item If $s=-1$
then $H^i(Z_{\ell},f^*_{\ell} L)=0$ for every $i$;
\item If $s\leq -2$
then $H^i(Z_\ell,f^*_{\ell}L)=H^{i-1}(Z_{\ell[1]},{p_{\ell[1]}}_*f^*_{\ell}(L +(s+1)K_{{l_r}}))$  for every $i$ and the rank $-(s+1)$ bundle ${p_{\ell[1]}}_*f^*_{\ell}(L +(s+1)K_{l_r})$ has a filtration on $Z_{\ell[1]}$ with factors
$$f_{\ell[1]}^*\big(L+tK_{l_r}\big ),\,\,\,t=s+1,\dots,-1.$$
\end{enumerate}
\end{lemma}

\begin{proof}

The first part of (1) follows from the fact that the hypothesis $s\geq 0$ implies that $R^i{p_{\ell[1]}}_*f^*_\ell L=0$ for all $i> 0$. For the second part, note that the statement is clear for $s=0$. In the case $s\geq 1$ we consider first the exact sequence on $Z_{\ell[1]}$:
$$
0\lra  f^*_{\ell[1]}K_{l_r} \longrightarrow {p_{\ell[1]}}_*H_r \longrightarrow \cO_{Z_{\ell[1]}}\lra  0.
$$
By \cite[Ex.II.5.16.(c)]{Ha}, there exists a filtration of the rank $(s+1)$ vector bundle ${p_{\ell[1]}}_*sH_r$  whose factors are multiples of $f^*_{\ell[1]}K_{l_r}$. Twisting this filtration with $f_{\ell[1]}^*L$ we obtain, by Remark \ref{rem:descent}, a filtration of ${p_{\ell[1]}}_*f^*_\ell L$, whose factors are as stated.

Statement (2) is a direct consequence of Lemma \ref{lem:leray}, while (3)
follows from the same Lemma and case (1).
\end{proof}

For nef line bundles on $X$ we may state the following vanishing theorem, consequence of the Kodaira-Kawamata-Viehweg Theorem:

\begin{lemma}\label{lem:j1}
With the same notation as above, let $D$ be a nef divisor on $Z_\ell$. Then $H^i(Z_\ell,f_\ell^*(K_{X}/2)+D)=0$ for $i >0$. In particular $H^i(Z_\ell, L)=0$ for every nef line bundle $L$ on $Z_\ell$ and every  $i > 0$.
\end{lemma}

\begin{proof}
Following  \cite[Lemma~6.1]{LT}, one can show that there exist rational numbers $0<\varepsilon_j \le 1$ such that the $\Q$-divisor $\sum_{j=1}^{r}  \varepsilon_j {Z}_{\ell(j)}$ is ample on ${Z}_\ell$,
hence the $\Q$-divisor $M=D + \sum_{j=1}^{r}  \varepsilon_j {Z}_{\ell(j)}$ is ample, too. Set $a_j:=1 -\varepsilon_j$. Then, by Lemma \ref{lem:canonical}  we may write $f_\ell^*(K_{X}/2)+D\equiv K_{Z_\ell}+ M+\sum a_j {Z}_{\ell(j)} $ and the vanishing follows from \cite[Theorem~2.64]{KM}.
\end{proof}

Later on we would use the following stronger vanishing result, which in the rational homogeneous case is a special version of Kumar's vanishing \cite[Theorem~8.1.8]{Ku}:

\begin{lemma}\label{lem:j3}
With the same notation as above, being $\ell=(l_1,\dots,l_r)$, then $H^i(Z_\ell, L-Z_{\ell(r)})=0$ for every nef line bundle $L$ on $Z_\ell$ and every  $i > 0$.
\end{lemma}

\begin{proof}
This proof has been taken from \cite[Theorem~16]{Ku2}.
Note first that there exist positive rational numbers $\varepsilon_1,\dots, \varepsilon_{r-1}$ such that $\sum_{j=1}^{r-1}\varepsilon_j Z_{\ell(j)}$ is the pull-back of an ample $\Q$-divisor in $Z_{\ell[1]}$. Hence, since $f_\ell^*(K_{X}/2)$ is nef and has degree one on $\gamma_r$, it follows that $A:=\sum_{j=1}^{r-1}\varepsilon_j Z_{\ell(j)}-f_\ell^*(K_{X}/2)$ is an ample divisor in $Z_\ell$.

Then, by Lemma \ref{lem:canonical}, we may write:
$$
L-Z_{\ell(r)}=K_{Z_{\ell}}+\sum_{j=1}^{r-1}Z_{\ell(j)}-f_\ell^*(K_{X}/2)+L,
$$
and, setting $E:=1+\sum_{j=1}^{r-1}\varepsilon_j\in\Q$, and $\Delta:=\sum_{j=1}^{r-1}\left(1-\frac{\varepsilon_j}{E}\right)Z_{\ell(j)}$, we have:
$$
L-Z_{\ell(r)}=K_{Z_{\ell}}+\Delta+L+\sum_{j=1}^{r-1}\left(\frac{\varepsilon_j}{E}\right)Z_{\ell(j)}-f_\ell^*(K_{X}/2).
$$
Since $L+\sum_{j=1}^{r-1}\left(\frac{\varepsilon_j}{E}\right)Z_{\ell(j)}-f_\ell^*(K_{X}/2)= L+\frac{1}{E}\big(A-
(E-1)
f_\ell^*(K_{X}/2)\big)$ is ample and $0<1-\frac{\varepsilon_j}{E}\leq 1$ for every $j$, the proof is concluded by \cite[Theorem~2.64]{KM}.
\end{proof}

Let $\Z[\Pic(X)]$ be the group algebra of $\Pic(X)$, in which, following \cite{De} we will denote by $e^L$ the element of $\Z[\Pic(X)]$ corresponding to $L$ under the canonical inclusion. For any reflection $r_i \in W$ let us define the  {\it Demazure operator}
$$D_{i}: \Z[\Pic(X)] \to \Z[\Pic(X)] $$ by setting, for any $L$ in $\Pic(X)$
$$D_{i}(e^L) = \begin{cases} e^L + e^{L+K_i} + \dots + e^{r_i(L)}
& \text{if~} L \cdot \Gamma_i \ge 0\\
0 & \text{if~}L \cdot \Gamma_i = -1\\
-e^{L-K_i} - e^{L- 2K_i} - \dots - e^{r_i(L)+K_i}
& \text{if~}L \cdot \Gamma_i \le -2
\end{cases}
$$
and extending it linearly. Note that, by definition, $D_i(e^L)=-D_i(e^{r_i(L)+K_i})$ for all $L\in\Pic(X)$.
Given a list $\ell=(l_1,\dots,l_r)$, $l_s\in I$ we define $D_\ell(e^L)=D_{{l_1}}(  \dots (D_{{l_r}}(e^L)))$.
\begin{notation}\label{notn:exp}
Let us  set, for every element $\Sigma_{j} n_je^{L_j}\in\Z[\Pic(X)]$:
\begin{eqnarray*}
H^i\Big(Z_\ell, \sum_{j} n_je^{L_j}\Big) &:= & \bigoplus_{j}n_jH^i(Z_\ell,  f^*_\ell L_j),\,\,\,i\in\Z,\\
h^i\Big(Z_\ell, \sum_{j} n_je^{L_j}\Big) &:= & \dim H^i\Big(Z_\ell, \sum_{j} n_je^{L_j}\Big),\,\,\,i\in\Z,\\
\Chi\Big(Z_\ell, \sum_{j} n_je^{L_j}\Big) &:= & {\sum}_{j}n_j\Chi(Z_\ell,  f^*_\ell L_j),\\
\deg \Big(\sum_{j} n_j e^{L_j} \Big)&:=& \sum n_j.
\end{eqnarray*}
\end{notation}
We then have that Lemma \ref{lem:j2} yields

\begin{proposition}\label{prop:Chi}
Let $L$ be a line bundle on $X$. Then
$$\Chi(Z_\ell, e^L) = \deg(D_{\ell}(e^{ L}))=\Chi(\overline Z_\ell, \overline e^{\psi (L)}).$$
If moreover $L$ is nef, then
$$h^0(Z_\ell, e^L) =h^0(\overline Z_\ell, \overline e^{\psi (L)});$$
in particular $f^*_{\ell} L$ is nef and big if and only if $\overline{f}^*_{\ell} \psi (L)$ is nef and big.
\end{proposition}

\begin{proof} By Lemma \ref{lem:j2}, for any line bundle $L \in \Pic(X)$, our previous definitions allow us to write:
$$
\Chi(Z_\ell,e^L)=\Chi(Z_\ell,f_\ell^*L)=\Chi(Z_{\ell[1]},D_{l_r}(e^L)),
$$
so that, recursively,  $\Chi(Z_\ell,e^L)=\Chi(Z_{\ell[r]},D_{\ell}(e^L))$ and, since $\Chi(Z_{\ell[r]}, f^*_{\ell[r]} L') =1$ for every $L' \in \Pic(X)$, we obtain the first equality. Since $\deg(D_{\ell}(e^{ L})) $ depends only on intersection numbers with the curves $\Gamma_i$, the second equality follows.
If $L$ is nef the higher cohomology of $f^*_{\ell} L$ and $\overline f^*_{\ell} \psi (L)$ vanishes
by Lemma \ref{lem:j1}. Finally the last statement follows from \cite[Lemma~2.2.3]{Laz}.
\end{proof}

\begin{corollary}\label{cor:stein2}
Let $\ell=(l_1,\dots,l_r)$ be a sequence. Then $\dim f_\ell(Z_\ell) =\dim Z_\ell$ if and only if $w(\ell)$ is reduced.
\end{corollary}

\begin{proof}
Let us consider the homogeneous model $G/B$ of $X$ and the corresponding Bott-Samelson variety $\overline{Z}_\ell$, with evaluation $\overline{f}_\ell:\overline{Z}_\ell\to G/B$. It is known that the property holds for $\overline{f}_\ell$, since $ \overline f_\ell(\overline Z_\ell)$ is the Schubert variety $\overline{Bw(\ell)B}/B$ of $G/B$, whose dimension is $\len(w(\ell))$  so that, in particular, the divisor $\overline{f}_\ell^*(-K_{G/B})$ is big if and only if $w(\ell)$ is reduced. Since $-K_{G/B}=\psi(-K_X)$ the result follows from
Proposition \ref{prop:Chi}.
\end{proof}

Consider a list $\ell=(l_1,\dots,l_r)$ satisfying that $r_{l_1}\circ\dots \circ r_{l_r}=w(\ell)$ is a reduced expression for the {\it longest element} of $W$. This means, by definition, that $w(\ell)$ is the unique element $w_0$ of maximal length in $W$ (see \cite[Sect. 1.8]{H2}). The next proposition shows that in this case the corresponding Bott-Samelson variety $Z_\ell$ is birational to $X$.

\begin{corollary}\label{cor:birat}
Let $\ell=(l_1,\dots,l_r)$ be a sequence such that $w(\ell)$ is a reduced expression of $w_0$. Then the morphism $f_\ell: Z_\ell \to X$ is surjective and birational.
\end{corollary}

\begin{proof}
The surjectivity of $f_\ell$ follows from Proposition~\ref{prop:dim} and Corollary \ref{cor:stein2}.
Let $L$ be an ample line bundle on $X$. By \cite[Lemma~3.3.3 (b)]{BK} it is enough to show that, for $s\gg0$, the restriction map $H^0(X,sL) \to H^0(Z_\ell, e^{sL})$, which is an injection by  the surjectivity of $f_\ell$, is an isomorphism.

On one hand, by Corollary \ref{cor:cohomequal} we have $h^0(X,sL)=h^0(G/B,\psi(sL))$. On the other, since $\overline{f}_\ell$ is birational (Cf. \cite[Proposition~2]{De}), it follows that $h^0(G/B,\psi(sL))=h^0(\overline{Z}_\ell, \bar e^{\psi(sL)})$. The proof is then finished by Proposition \ref{prop:Chi}.
\end{proof}

We finish this section by noting that the birationality of a dominant evaluation map $f_\ell$ allows us to reduce Theorem \ref{conj:CPforFT} to the case in which the diagram $\cD$ is connected.

\begin{corollary}\label{cor:product}
Assume that  $\cD = \cD_1 \sqcup \cD_2$. Then $X \simeq X_1 \times X_2$, where $X_1$ and $X_2$ are Fano manifolds whose elementary contractions are smooth $\P^1$-fibrations and whose Dynkin diagrams are $\cD_1$ and $\cD_2$, respectively.
\end{corollary}

\begin{proof}
As in Lemma \ref{lem:prodBS} we consider two sequences $\ell_1$ and $\ell_2$ on the indices $I_1, I_2\subset I$ determined by $\cD_1$ and $\cD_2$, so that, $Z_{\ell_1\ell_2}\simeq Z_{\ell_1}\times Z_{\ell_2}$. We will assume that they provide reduced expressions for the longest elements of the Weyl groups of $\cD_1$ and $\cD_2$, respectively, so that the evaluation $f_{\ell_1\ell_2}:Z_{\ell_1\ell_2}\to X$ is birational.

In other words, $f_{\ell_1\ell_2}$ is the contraction determined by the extremal face of $\cNE{Z_{\ell_1\ell_2}}$ generated by $R:=\{\gamma_i|\,\,l_i=l_k\mbox{ for some } k>i\}$ (cf. Corollary \ref{cor:stein}). This face is the convex hull of the two extremal faces generated by the sets $\{\gamma_i\in R|\,\,i\in I_1\}$, $\{\gamma_i\in R|\,\,i\in I_2\}$, that provide contractions of $h_{\ell_1}:Z_{\ell_1}\to X_1$ and $h_{\ell_2}:Z_{\ell_2}\to X_2$, satisfying that $X\simeq X_1\times X_2$ and that, up to this isomorphism $f_{\ell_1\ell_2}=h_{\ell_1}\times h_{\ell_2}$. In particular, $X_1$ and $X_2$ are smooth Fano manifolds.

Finally, for every elementary contraction $f:X_1\to X'_1$, its pullback $X\simeq X_1\times X_2\to X_1'\times X_2$ is an elementary contraction of $X$, hence it follows that $f$ is a smooth $\P^1$-fibration. The same holds for $X_2$.
\end{proof}

\section{Homogeneity - General case}\label{sec:unique}

With the same notation as in the previous sections, we will further assume here (cf. Corollary \ref{cor:product}) that the Dynkin diagram $\cD$ of $X$ is connected, hence of type ${\rm A}_n$, ${\rm B}_n$, ${\rm C}_n$, ${\rm D}_n$, ${\rm E}_6$, ${\rm E}_7$, ${\rm E}_8$, ${\rm F_4}$ or ${\rm G}_2$. We start by noting that the case ${\rm G}_2$ follows from previous results:

\begin{lemma}\label{lem:G2}
Let $X$ be a Fano manifold whose elementary contractions are smooth $\P^1$-fibrations. If its Dynkin diagram is of type ${\rm G}_2$, then $X$ is isomorphic to its homogeneous model $G/B$.
\end{lemma}

\begin{proof}
Since the Cartan matrix of ${\rm G}_2$ has determinant one, there exist two Cartier divisors $L_1$ and $L_2$ on $X$ satisfying that $L_i\cdot\Gamma_j=\delta_{i}^j$. In particular the elementary contractions of $X$ are $\P^1$-bundles, and then the statement follows from \cite[Theorem~1.1]{Wa}.
\end{proof}

We may then assume that $\cD$ is different from ${\rm G}_2$. Our strategy will be to compare first the Bott-Samelson varieties $Z_\ell$ of  $X$ with the Bott-Samelson varieties $\overline{Z}_\ell$ of its homogeneous model $G/B$. In fact, we will show (cf. Propositions \ref{prop:uniqueADE} and \ref{prop:uniqueBC}) that, for a suitable choice of a sequence $\ell=(l_1, \dots,l_r)$ corresponding to a reduced expression of the longest element in $W$, we have $$Z_{\ell[s]} \simeq \overline Z_{\ell[s]}\mbox{ for every }s = 0, \dots, r-1.$$ %\par
%\medskip

In order to describe explicitly the sequences $\ell$ we are going to use,
let us introduce the following notation:
$$
u_i=(1,2,\dots,i),\quad d_i=(i,i-1\dots,1),\mbox{ for every }i=1,\dots,n.
$$
Ordering the nodes of $\cD$ as in \cite[page 58]{H}, the table %\ref{tab:reduced}
below lists the sequences we are going to use in each case. Whereas in the case of a Dynkin diagram without multiple edges we show our uniqueness result for any sequence $\ell$,
in the cases ${\rm B}_n$ and ${\rm C}_n$, $\ell$ was chosen conveniently so that the uniqueness property holds. In the case of ${\rm F}_4$ we have checked (using the software system {\tt Sage}) that our proof, based on cohomological computations, does not work for any of the $2144892$ possible sequences providing a reduced expression of the longest word. This is the reason why the proof of Theorem \ref{conj:CPforFT} in case $\cD = {\rm F}_4$ is different; we will present it in Section \ref{sec:F4}.

\begin{table}[!!h]
\centering
\begin{tabular}{|c||c|c|c|c|}\hline
$\cD$&${\rm A}_n,{\rm D}_n,{\rm E}_n$&${\rm B}_n$&${\rm C}_n$&${\rm F}_4$\\\hline
$\ell$&any&$(u_n)^n$&$(d_n)^n$&none\\\hline
\end{tabular}\label{tab:reduced}
\end{table}

\subsection{Uniqueness of Bott-Samelson varieties}\label{ssec:unique}

Throughout the section we will use the notation for cohomology of line bundles introduced in \ref{notn:exp}.
Let us start by stating a corollary of Lemma \ref{lem:j1}, which easily follows from the fact that $K_{X}/2 \cdot \Gamma_j=-1$ for every $j= 1, \dots, n$.

\begin{corollary}\label{cor:KKV2} Let $Z_\ell$ be the Bott-Samelson variety of $X$ defined by the sequence
$\ell=(l_1,\dots,l_r)$,
%be a se\-quence defin\-ing the Bott-Samelson variety $Z_\ell$
and $L$ be a line bundle on $X$ satisfying that $L \cdot \Gamma_{l_j} \ge -1$ for every $j= 1, \dots, r$. Then $H^i(Z_\ell, e^L)=0$ for $i>0$. In particular, if the node $j$ and the node $j'$ are connected then $H^i(Z_\ell,e^{K_j+K_{j'}})=0$ for $i>0$.
\end{corollary}

\begin{proof}
It is enough to note that if $L$ satisfies the required property, then, by Corollary \ref{cor:coneBS}, $f_\ell^*(L-K_X/2)$ is nef on $Z_\ell$. Then the result follows by Lemma \ref{lem:j1}.
\end{proof}

The following result gives a set of descent rules for cohomology of line bundles on Bott-Samelson varieties, which are consequences of  Lemma \ref{lem:j2}:

\begin{lemma}\label{lem:descent}
Let $\ell=(l_1,\dots,l_r)$ be the sequence defining the Bott-Samelson variety $Z_\ell$, and let $L$ be a line bundle on $X$, of degree $s$ with respect to $\Gamma_{l_r}$. Then:
\begin{enumerate}
\item[(DR1)] If $s=0$, then $H^i(Z_\ell,e^{L})=H^i(Z_{\ell[1]},e^{L})$ for all $i$.
\item[(DR2)] If $s=-1$, then $H^i(Z_\ell,e^L)=0$ for all $i$.
\item[(DR3)] If $s=-2$, then $H^i(Z_\ell,e^{L})=H^{i-1}(Z_{\ell[1]},e^{L-K_{l_r}})$, for all $i$.
\item[(DR4)] If $s\geq 1$, then $h^i(Z_\ell,e^L)\leq h^i(Z_{\ell[1]},D_{l_r}(e^{L}))$ for all $i$.
\item[(DR5)] If $s=1$ and $(L+K_{l_r})\cdot\Gamma_{l_j}\geq -1$ for every $j=1,\dots,r-1$, then $H^i(Z_\ell,e^{L})=H^i(Z_{\ell[1]},e^{L})$ for all $i>0$. In particular, this holds for $L=K_j$, whenever $K_j\cdot\Gamma_{l_r}=1$.
\end{enumerate}
\end{lemma}

\begin{proof}
Items (DR1) to (DR4) follow directly from Lemma \ref{lem:j2} and, eventually, the corresponding long exact sequences of cohomology associated to the filtrations provided there. For (DR5) one needs to use also Corollary \ref{cor:KKV2}.
\end{proof}

In particular, we get the following:

\begin{proposition}\label{prop:descent}
Let $\ell=(l_1, \dots, l_r)$ be a sequence in $I$, and $j\in I$, $k >0$ be two integers such that
 $K_j \cdot \Gamma_{l_{i}} =0,1$ for every $i > r-k$;
then $h^1(Z_\ell, e^{K_j})= h^1(Z_{\ell[k]}, e^{K_j})$.
\end{proposition}

\begin{proof} If $k=1$, then the result follows by (DR1) if $K_j \cdot \Gamma_{l_{r}}=0$ and by (DR5) if $K_j \cdot \Gamma_{l_{r}}=1$. The general statement follows by recursion.
\end{proof}

\begin{corollary}\label{cor:simple}
Let $\ell=(l_1, \dots, l_r)$ be a list in $I$ satisfying that $K_{l_r} \cdot \Gamma_{l_i} \in \{0,1,-2\}$ for every $i<r$, and set $J:=\{i<r|\,\,l_i=l_r\}$.
Then $H^1(Z_{\ell[1]}, e^{K_{l_r}})$ is generated by $\zeta_\ell$, and this cocycle is zero if and only if $J= \emptyset$.
\end{corollary}

\begin{proof} If $J=\emptyset$, then, by Proposition \ref{prop:descent},
$H^1(Z_{\ell[1]}, e^{K_{l_r}})$ is isomorphic to $ H^1(Z_{\ell[r-1]},e^{K_{l_r}})= H^1(\Gamma_{l_1}, K_{l_r}|_{\Gamma_{l_1}})=0$.

If $J\neq\emptyset$, denote by $r-s$ its maximum. Then Proposition \ref{prop:descent} tells us that
$H^1(Z_\ell, e^{K_{l_r}}) \cong H^1(Z_{\ell[s]},e^{K_{l_r}})$. On the other hand
(DR3) provides $h^{1}(Z_{\ell[s]}, e^{K_{l_r}})=h^0(Z_{\ell[s+1]}, e^{K_{l_r}-K_{l_r}})=1$. Since, by Corollary \ref{cor:nonsplit}, $\zeta_\ell\in H^{1}(Z_{\ell[1]}, e^{K_{l_r}})$ is nonzero, the claim follows.
\end{proof}

We conclude from this the following uniqueness result for Bott-Samelson varieties in the case in which $\cD$ has no multiple edges:

\begin{proposition}\label{prop:uniqueADE} Assume that $\cD = {\rm A}_n, {\rm D}_n$ or ${\rm E}_n$. Then, for every sequence $\ell=(l_1, \dots, l_r)$, the manifold
$Z_{\ell}$ is isomorphic to $\overline Z_{\ell}$.
\end{proposition}

\begin{proof}
We will use induction on $r$, noting that the result is obvious for $r=1$. Then, given $\ell$, we may assume that $Z_{\ell[1]}\simeq\overline Z_{\ell[1]}$. It is then enough  to show that the cocycles $\zeta_\ell,\overline \zeta_\ell\in H^1(Z_{\ell[1]}, e^{K_{l_r}})$ defining $Z_{\ell}$ and $\overline Z_{\ell}$, respectively, are proportional, up to this isomorphism. Since $K_{l_r} \cdot \Gamma_{l_i} \in \{0,1,-2\}$ for every $i<r$, by the assumptions on $\cD$, this follows by Corollary \ref{cor:simple}.
\end{proof}

The case in which $\cD={\rm B}_n$ or ${\rm C}_n$ is more involved. Let us recall that in this situation, with the standard ordering (cf. \cite[page 58]{H}), $K_i\cdot \Gamma_j\in\{0,1,-2\}$ unless $(i,j)=(n-1,n)$, in the case ${\rm B}_n$, or $(i,j)=(n,n-1)$, in the case $C_n$. Hence, in order to obtain a uniqueness result, we will need to control cohomology groups of the type $H^1(Z_\ell,e^{K_{n-1}})$ and $H^1(Z_\ell,e^{K_{n}})$, for ${\rm B}_n$ and ${\rm C}_n$, respectively, and some particular choices of $\ell$. The precise statement we are going to prove is the following:

\begin{lemma}\label{lem:uniqueBC} Assume that $\cD = {\rm B}_n$ (resp. ${\rm C}_n$) and, given $k\in\{ 1,\dots, n-1\}$, consider the sequence $\ell=(u_n)^k$ (resp. $\ell=(d_n)^k$). Then $H^1(Z_\ell,e^{K_{n-1}})$ (resp. $H^1(Z_\ell,e^{K_{n}})$) is $1$-dimensional.
\end{lemma}
\begin{proof}
Let us consider first the case $\cD = {\rm B}_n$. By Corollary~\ref{cor:nonsplit}, it is enough to prove $h^1(Z_\ell,e^{K_{n-1}})\leq 1$. We start by noting that  $h^1(Z_{\ell[1]},e^{K_{n-1}})=1$ by (DR3), and that $h^1(Z_{\ell[1]},e^{K_{n-1}+K_n})=0$ by Corollary \ref{cor:KKV2}. Then (DR4) provides:
\begin{equation*}
h^1(Z_{\ell}, e^{K_{n-1}})\leq 1+ h^1(Z_{\ell[1]}, e^{K_{n-1}+2K_n})=1+ h^1(Z_{\ell[2]}, e^{K_{n-1}+2K_n}),
\end{equation*}
where the last equality follows by (DR1), since $(K_{n-1}+2K_n)\cdot\Gamma_{n-1}=0$.

Set $D_i:= \sum_{j=i}^{n-1} K_j + 2K_n$ for every $i=1, \dots, n-1$. The proof will be finished by showing that $h^1(Z_{\ell[2]}, e^{D_{n-1}})=0$.

Note that, since $D_i\cdot \Gamma_{i-1}=1$ and $D_i\cdot \Gamma_j=0$ for any $i=2,\cdots n-1$ and any $j< i-1$, (DR4) and (DR1) provide
\begin{equation*}
\begin{array}{lll}\vspace{0.2cm}
h^1\big(Z_{\ell[n-i+1]}, e^{D_{i}}\big) &\le& h^1\big(Z_{\ell[n-i+2]}, e^{D_{i}}+e^{D_{i-1}}\big)\\
&\le& h^1\big(Z_{\ell[n]}, e^{D_{i}}\big) +h^1\big(Z_{\ell[n-i+2]}, e^{D_{i-1}}\big),
\end{array}
\end{equation*}
for every $i=2,\cdots, n-1$. Hence, recursively we obtain:
\begin{equation*}
h^1\big(Z_{\ell[2]}, e^{D_{n-1}}\big)\leq\sum_{i=1}^{n-1}h^1\big(Z_{\ell[n]}, e^{D_{i}}\big).
\end{equation*}

Note that $\ell[n]=(u_n)^{k-1}$, hence the right hand side of this inequality is obviously zero if $k=1$. In the case $k>1$, since $D_i\cdot \Gamma_{n}=-2$ for every $i$, (DR3) provides:
$$
h^1\big(Z_{\ell[n]}, e^{D_{i}}\big)  =   h^0\big(Z_{\ell[n+1]}, e^{D_{i}-K_{n}}\big).
$$
Since $(D_i-K_n)\cdot\Gamma_j=0$ for $j=i+1,\cdots, n-1$, by (DR1) we have
$$
h^0\big(Z_{\ell[n+1]}, e^{D_{i}-K_{n}}\big)=h^0\big(Z_{\ell[2n-i]}, e^{D_{i}-K_{n}}\big),
$$
and this group is zero by (DR2).

The case $\cD = {\rm C}_n$ follows by a similar argument, in which we will need to consider the line bundles $D'_i:= \sum_{j=i}^{n-1} 2K_j + K_n$, for $i=1,\dots,n-1$. Note that $D'_i\cdot\Gamma_j$ equals $-2$ if $j=i$, $2$ if $j=i-1$ and zero otherwise. In this situation (DR1), (DR4) and (DR3), together with Corollary \ref{cor:KKV2}, provide:
\begin{equation*}
\begin{array}{lcl}\vspace{0.2cm}
h^1\big(Z_{\ell}, e^{K_{n}}\big)\hspace{-0.2cm}&=&\hspace{-0.2cm}h^1\big(Z_{\ell[n-2]}, e^{K_{n}}\big)\leq
h^1\big(Z_{\ell[n-1]}, e^{K_{n}}+e^{K_{n-1}+K_n}+e^{D'_{n-1}}\big)\\
&\leq &1+ h^1\big(Z_{\ell[n-1]}, e^{D'_{n-1}}\big),
\end{array}
\end{equation*}
and we may conclude by showing that $h^1(Z_{\ell[n-1]}, e^{D'_{n-1}})=0$.
We claim first that
\begin{equation}\label{eq:Cn} h^1\big(Z_{\ell[n-1]}, e^{D'_{n-1}}\big)\leq h^1\big(Z_{\ell[i(n-1)]}, e^{D'_{n-i}}\big),\mbox{ for all }i=2,\dots,k.
\end{equation}
We may write, by (DR1), (DR4) and Corollary \ref{cor:KKV2}:
$$
\begin{array}{lcl}\vspace{0.2cm}
h^1\big(Z_{\ell[(i-1)(n-1)]}, e^{D'_{n-i+1}}\big)&=&h^1\big(Z_{\ell[i(n-1)-1]}, e^{D'_{n-i+1}}\big)\\\vspace{0.2cm}
&\leq &h^1\big(Z_{\ell[i(n-1)]}, e^{D'_{n-i+1}}
+e^{D'_{n-i+1}+2K_{n-i}})\big.
\end{array}
$$
Since, by (DR3), $h^1(Z_{\ell[i(n-1)]}, e^{D'_{n-i+1}})=h^0(Z_{\ell[i(n-1)+1]}, e^{D'_{n-i+1}-K_{n-i+1}})$, and this is zero
by (DR2), the claim follows.

Now, using the inequality (\ref{eq:Cn}), together with (DR1)
we may finally write:
$$
h^1\big(Z_{\ell[n-1]}, e^{D'_{n-1}}\big) \leq h^1\big(Z_{\ell[k(n-1)]}, e^{D'_{n-k}}\big)=h^1\big(Z_{\ell[kn]}, e^{D'_{n-k}}\big),
$$
which is trivially equal to zero. This finishes the proof.
\end{proof}

We may now state the following uniqueness result:

\begin{proposition}\label{prop:uniqueBC}
Assume that $\cD = {\rm B}_n$ (resp. ${\rm C}_n$) and consider the sequence $\ell=(u_n)^n$ (resp. $(d_n)^n$). Then, for every $i$,
%$i\leq n^2$,
the group $H^1(Z_{\ell[i]},e^{K_n})$ is generated by $\zeta_{\ell[i-1]}$, and this cocycle is zero unless $i<n(n-1)$. In particular, $Z_{\ell[i]}\simeq \overline Z_{\ell[i]}$ for every $i$.
\end{proposition}

\begin{proof}
We will show the proof for $\cD={\rm B}_n$, being the case $\cD={\rm C}_n$ analogous. We proceed by induction on $i$. Let us assume that $Z_{\ell[i+1]}\simeq \overline Z_{\ell[i+1]}$ and consider the cocycles $\zeta_{\ell[i]},\overline\zeta_{\ell[i]} \in H^1(Z_{\ell[i+1]}, e^{K_{l_{n^2-i}}})$. If $l_{n^2-i}\neq n-1$, then the result follows from Corollary \ref{cor:simple}. If $l_{n^2-i}=n-1$, then let $k$ denote the integral part of $(n^2-i)/n$ and $\ell'=(u_n)^k$. If $k=0$, then $H^1(Z_{\ell[i+1]}, e^{K_{n-1}})=0$ by Corollary \ref{cor:simple}. If $k>0$, Proposition \ref{prop:descent} and Lemma \ref{lem:uniqueBC} provide:
$$
H^1\big(Z_{\ell[i+1]}, e^{K_{n-1}}\big)\simeq H^1\big(Z_{(u_n)^k}, e^{K_{n-1}}\big)\simeq \C.
$$
Since  $\zeta_{\ell[i]}$ and $\overline\zeta_{\ell[i]}$ are nonzero by Corollary \ref{cor:nonsplit}, they are necessarily proportional, from which we get $Z_{\ell[i]}\simeq \overline Z_{\ell[i]}$.
\end{proof}

\begin{proof}[of Theorem \ref{conj:CPforFT} for $\cD \not = {\rm F}_4$]
The case $\cD={\rm G}_2$ has been considered in Lemma \ref{lem:G2}.
Note also that $\dim X$ is equal to the dimension of its homogeneous model $G/B$, by Proposition \ref{prop:dim}. In particular, the assumption $\dim(X)\neq 24$ implies that $\cD$ is different from ${\rm F}_4$.

For $\cD\neq {\rm G}_2,{\rm F}_4$ we know, from Propositions \ref{prop:uniqueADE} and \ref{prop:uniqueBC}, that there exists a sequence $\ell=(l_1, \dots, l_r)$ satisfying that $w(\ell)=r_{l_1}\circ\dots\circ r_{l_r}$ is a reduced expression of the longest element of $W$, and such  that $Z_\ell=\overline{Z}_\ell$.
The two maps $f_{\ell}:Z_\ell\to X$ and $\overline{f}_\ell:Z_\ell\to G/B$ are both birational (cf. Corollary \ref{cor:birat}), and
they both have the same Stein factorization, by Corollary \ref{cor:stein}.
It follows that $X\simeq G/B$.\end{proof}

\section{Manifolds whose Dynkin diagram is ${\rm F_4}$}\label{sec:F4}

In this section
we will prove Theorem \ref{conj:CPforFT} for the case $\cD={\rm F}_4$, by following the line of argumentation of \cite[Section 6]{MOSW}. More concretely,
numbering the nodes of $\cD$ as in \cite[p.~58]{H},
we will consider the contraction $\pi^1:X \to X^1$ associated with the three-dimensional face of $\NE(X)$ generated by the rays $R_i$ $i \not = 1$; we will first show that $X^1$ is isomorphic to the corresponding rational homogeneous space by comparing their varieties of minimal rational tangents (see definition below), and then we will reconstruct inductively the variety $X$ upon $X^1$.

Let us briefly recall first some well-known facts on varieties of minimal rational tangents (cf. \cite{Hw2}), and introduce some notation.

\begin{notation}\label{notn:VMRT}
Given a smooth projective variety $Y$, a family
of rational curves on $Y$ is an irreducible component of the
scheme $\rat^n(Y)$ (see \cite[II.2]{Ko}). We say that a family
$\cH$ is {\it unsplit} if $\cH$ is a proper $\C$-scheme.

For a general point $y\in Y$, let $\cH_y$ be the normalization of the family of $\cH$-curves passing through $y$, and consider the corresponding universal family:
\[\xymatrix@=25pt{
 \cU_y \ar[d]^{\pi_y}  \ar[r]^{\iota_y}  &Y   \\
 \cH_y \ar@/^/[u]^{\sigma_{y}} &  \\
} \]
where $\sigma_y$ is the unique section of $\pi_y$ such that $\iota_y(\sigma_y(\cH_y))=\{y\}$.
Setting $\cE_y=(\pi_y)_*\cO_{\cU_y}(\sigma_y)$ we have that $\cU_y \simeq \P(\cE_y)$ and $\cO_{\cU_y}(\sigma_y)$ is the tautological line bundle.
A member of $\cH_y$ is called an {\it $\cH_y$-curve}.
We define the tangent map ${\tau}_y : {\cH}_y \rightarrow \P({T_yY}^{\vee})$ by assigning the tangent vector at $y$ to each $\cH_y$-curve which is smooth at $y$. We denote by $\cC_y \subset \P({T_yY}^{\vee})$ the closure of the image of ${\tau}_y$, which is called the {\it VMRT}    (\textit{\!Variety of Minimal Rational Tangents}) at $y$.
By \cite[Sequence 2.6]{AC},  $\tau_y^* \cO_{\cC_y}(1)= c_1(\cE_y)$.\par
\end{notation}

The VMRT's of a Fano manifold are not necessarily smooth. Later on we will need to use the following smoothness criterion.

\begin{lemma}\label{lem:smoothVMRT} Let $Y$ be a Fano manifold of Picard number one and denote by $\cO_Y(1)$ the ample generator of $\Pic(Y)$. Assume that $\cO(1)$ is spanned and that the rational curves of a family $\cH$ have degree one with respect to $\cO_Y(1)$.
Then, at a general point $y \in Y$ the variety of minimal rational tangents $\cC_y$ of $\cH$ is smooth.
\end{lemma}

\begin{proof} Note first that the hypotheses imply that the family $\cH$ is unsplit and that all its elements are smooth rational curves in $Y$. Hence we may apply \cite[Theorem~1.3]{KK}, to get that the tangent map $\tau_y$ is injective. According to \cite[Proposition~1.4]{Hw2}, it is then enough to show that every $\cH_y$-curve $C$ is standard, that is, $T_{Y}|_C \cong \cO_{\P^1}(2)\oplus \cO_{\P^1}(1)^{\oplus a}\oplus \cO_{\P^1}^{\oplus b}$, for some $a,b$.

To prove this, let $\psi: Y \rightarrow \P^d$ be a surjective morphism defined by a linear system in $|\cO_Y(1)|$, where $d:=\dim Y$.
The image via  $\psi$ of  $C$ is a line in $\P^d$ and, by the generality of $y \in Y$, $C$ is not contained in the ramification locus of $\psi$. In particular, the natural morphism of normal bundles $N_{C/Y} \rightarrow N_{\psi(C)/\P^d} \cong \cO_{\P^1}(1)^{\oplus d-1}$ is generically surjective. Since moreover, the generality of $C$ also implies that $N_{C/Y}$ is globally generated, it follows that $C$ is standard.
\end{proof}

\begin{notation}\label{notn:f4}
From now on $X$ will be a Fano manifold whose elementary contractions are smooth $\P^1$-fibrations, with
Dynkin diagram ${\rm F}_4$. Besides the notation introduced in the previous sections of the paper, we will denote by $R_J$, $J \subset I:=\{1,2,3,4\}$, the extremal face of $\cNE{X}$ spanned by the rays $R_j$, $j \in J$. Its associated contraction will be denoted either by $\pi_J:X \to X_J$, or by
$\pi^{J^c}:X \to X^{J^c}$, where $J^c$ is the complement of $J$ in $I$.
Moreover, we will denote
by $\cD_J$ the Dynkin diagram obtained from $\cD$ by deleting the nodes
which are not in $J$ and by $W_J$ the Weyl group of $\cD_J$.
Finally, we will denote by $\overline{X}$ the homogeneous model $G/B$ of $X$, and we will add an overline to the usual notation ($\overline{R}_J$, $\overline{\pi}_J$,\dots) to denote the corresponding extremal faces, contractions, etc.
\end{notation}

\subsection{Smoothness of the contractions of $X$}\label{ssec:smoothcont}

In this section we will show that, for every $J \subset I$, the contraction $\pi_J:X \to X_J$ is smooth,
and any fiber of $\pi_J$ is the complete flag manifold with Dynkin diagram $\cD_J$.

We will follow here the notation and ideas from \cite[IV. Section 4]{Ko}. For each $j = 1, \dots, 4$
 the proper prerelation given by $X\stackrel{s_j}{\longleftarrow}U_j\stackrel{w_j}{\longrightarrow}X$, and $X\stackrel{\sigma_j}{\longrightarrow}U_j$, fitting in the diagram:
$$
\xymatrix@=35pt{\hspace{-0.7cm}U_j:=X\times_{X_j}X\ar[r]^{w_j\hspace{-0.4cm}}\ar[d]^{s_j}&X\ar[d]^{\pi_{j}} \\
X\ar@/^/[u]^{\sigma_{j}} \ar[ru]^{\mbox{\tiny id}}\ar[r]^{\pi_{j}}&X_{j}}
$$
The scheme $U_j$, together with the maps $w_j,s_j$, is a proper algebraic relation on $X$.

Moreover, any list $\ell=(l_1, \dots, l_r)$ in $I$ defines an algebraic relation $U_\ell:=U_{l_1}*\dots *U_{l_r}$ (see \cite[IV. Definition 4.3]{Ko}) which is proper (\cite[IV. Lemma 4.3.1]{Ko}) and satisfies, by construction, that $(x,y)\in\widetilde{U_\ell}$ if and only if $y\in f_\ell(Z_\ell)$, being $Z_\ell$ the Bott-Samelson variety determined by the list $\ell$ and by $Z_{\ell[r]}=\{x\}$.

Finally, for every subset $J\subset I$ we may consider the proalgebraic relation $\ch(J)$ defined as the countable union of the relations $U_\ell$ constructed with lists $\ell$ of indices in $J$. By \cite[IV. Theorem 4.16]{Ko} there exists a proper fibration $\pi:X^0 \to Y^0$, defined on an open subset $X^0 \subset X$, whose fibers are $\ch(J)$-equivalence classes.

\begin{proposition}\label{prop:fiberschubert} Let $J$ be any nonempty subset of $I$. Let $x \in X$ be any point, let
$\ell=(l_1, \dots, l_r)$ be a reduced list such that $l_i \in J$ for all $i$ and $w(\ell)$ is the longest word in $W_J$, and let
$Z_\ell$ be the Bott-Samelson variety associated to $\ell$ such that $Z_{\ell[r]}=\{x\}$.
Then $f_\ell(Z_\ell)$ is set-theoretically equal to $\pi_J^{-1}(\pi_J(x))$.
\end{proposition}

\begin{proof}
Since any $\ch(J)$-equivalence class is contained in a fiber of $\pi_J$, clearly $f_\ell(Z_\ell)\subset \pi_J^{-1}(\pi_J(x))$, hence it is enough to show the opposite inclusion.

With the same notation as above, let $H_Y$ be a general very ample effective divisor on $Y$ such that $H_Y \cap Y_0 \not = \emptyset$. Set $H^0= \pi^{-1}(H_Y)$ and denote by $H$ the closure of $H^0$ in $X$. This divisor is trivial on curves $\Gamma_{j_1}, \dots, \Gamma_{j_k}$, hence on the
face $R_J$. It follows that a general fiber of $\pi_J$ is a $\ch(J)$-class.

But then for some list $\ell'$,
the morphism $U_{\ell'} \longrightarrow X \times_{X_J} X$
is dominant. Since the relation $U_{\ell'}$ is proper, the morphism is surjective and, in particular, $\pi_J^{-1}(\pi_J(x))$ is $\ch(J)$-connected.

Finally, let us show that $f_\ell(Z_\ell)=\pi_J^{-1}(\pi_J(x))$. If this is not the case, since $\pi_J^{-1}(\pi_J(x))$ is $\ch(J)$-connected, there exists $l_{r+1}\in J$ such that, being $\bar{\ell}=(l_1,\dots,l_r,l_{r+1})$, we have $f_\ell(Z_\ell)\subsetneq f_{\bar{\ell}}(Z_{\bar{\ell}})$. Since these two varieties are irreducible, $\dim (f_\ell(Z_\ell))< \dim( f_{\bar{\ell}}(Z_{\bar{\ell}}))$, contradicting the reducedness of $\ell$ by Corollary  \ref{cor:stein2}.
\end{proof}

\begin{proposition}\label{prop:surj} Let $\ell=(l_1, \dots, l_r)$ be a reduced  list in $I$, and $Z_\ell$ an associated Bott-Samelson variety. Then the map $H^0(f_\ell(Z_\ell), L|_{f_\ell(Z_\ell)}) \to H^0(Z_\ell, f^*_\ell L)$ is an isomorphism for any nef line bundle $L$ on $X$. In particular  $f_\ell(Z_\ell)$ is normal and $f_\ell$ is birational.
\end{proposition}

\begin{proof} Complete the list $\ell$ to a reduced expression $\bar \ell$ of the longest word $w(\bar \ell)\in W$. Consider the commutative diagram
$$
\xymatrix@=20pt{H^0(X,L)\ar[d] \ar[r]^{\sim} &H^0(Z_{\bar \ell}, e^L) \ar@{-{>>}}[d] \\
H^0(f_{\bar \ell[1]}(Z_{\bar \ell[1]}),L
%|_{f_{\bar \ell[1]}(Z_{\bar \ell[1]})}
)\,\,  \ar[d] \ar@{^{(}->}[r] &H^0(Z_{\bar \ell[1]}, e^L)\ar@{-{>>}}[d]\\
H^0(f_{\bar \ell[2]}(Z_{\bar \ell[2]}),L
%|_{f_{\bar \ell[1]}(Z_{\bar \ell[2]})}
)  \,\,  \ar[d] \ar@{^{(}->}[r]  &H^0(Z_{\bar \ell[2]}, e^L)\ar@{-{>>}}[d]\\
\dots & \dots\\
}
$$
where the top horizontal map is an isomorphism by the birationality of $f_{\bar \ell}$ (Corollary \ref{cor:birat}) and the other horizontal maps are injective by the surjectivity of $f_{\bar \ell[k]}:Z_{\bar \ell[k]} \to X_ {\bar \ell[k]}$.
By  Lemma \ref{lem:j3}, for any list $\ell'$ and  any nef line bundle $L$ on $X$ we have the vanishing $h^1(Z_{\ell'}, f^*_{\ell'} L - Z_{\ell'[1]})=0$, so the vertical maps on the right are surjective.
Then, recursively, the horizontal maps are isomorphisms, and we obtain the normality of $f_\ell(Z_\ell)$ from \cite[Lemma~ 3.3.3]{BK}. \end{proof}

\begin{proposition}\label{prop:fibersf4}
Let $J\subset I$ be a proper nonempty subset. Then every fiber of $\pi_J$ is the complete flag manifold with Dynkin diagram $\cD_J$. Moreover $\pi_J$ and  $X_J$ are smooth.
\end{proposition}

\begin{proof}  Let $x \in X$ be any point. By Proposition \ref{prop:fiberschubert} the fiber of $\pi_J^{-1}(\pi_J(x))$ passing through $x$
 equals (set-theoretically) $f_\ell(Z_\ell)$, where $Z_\ell$ is the Bott-Samelson variety constructed starting from $x$ and corresponding to a list $\ell=(l_1, \dots, l_r)$ such that  $w(\ell)$ is the longest word in $W_J$.
By Propositions \ref{prop:uniqueADE} and \ref{prop:uniqueBC}, $\ell$ can be chosen so that $Z_{\ell} \simeq \overline Z_\ell$, where $\overline{Z}_\ell$ is the corresponding Bott-Samelson variety of the homogeneous model $\overline{X}=G/B$ of $X$.
Since $f_\ell(Z_\ell)$ is normal and $f_\ell$ is birational by Proposition \ref{prop:surj}, the map $f_\ell$ is the contraction of $Z_\ell$ corresponding to the extremal face
generated by cycles $\{\gamma_i\,\,| \,\,l_i=l_k\mbox{ for some } k>i\}$ (see Corollary \ref{cor:stein}), hence
$f_\ell (Z_{\ell}) \simeq \overline f_\ell(\overline Z_\ell)$ is the complete flag manifold with Dynkin diagram $\cD_J$.

For the second part, by \cite[Lemma~4.13]{SW}, it is enough to prove that the normal bundle $N_{f_\ell (Z_{\ell})/X}$ of $f_\ell (Z_{\ell})$ in $X$ is trivial, for any $f_\ell (Z_{\ell})$.

In order to see this, we consider its restriction to any curve $\Gamma_j \subset f_\ell (Z_{\ell})$, $j\in J$, that is the cokernel of the natural inclusion $N_{\Gamma_j/f_\ell (Z_{\ell})} \lra N_{\Gamma_j/X}$. Since these two bundles are trivial, $(N_{f_\ell (Z_{\ell})/X})|_{\Gamma_j}$ is trivial as well.

Denote $Z:=f_\ell (Z_{\ell})$, $Y:=\P(N_{f_\ell (Z_{\ell})/X})$, and $\pi:Y\to Z$ the natural projection. Let $\tilde{\Gamma}_j$ denote a minimal section of $Y$ over $\Gamma_j$, for all $j$. By the previous argument, there is precisely one of these sections passing by every point of $Y$. For each $j\in J$, the family of deformations of $\tilde{\Gamma}_j$ is unsplit and dominates $Y$, hence we may consider the proper fibration $\varphi:Y^0\to P^0$, defined on an open set $Y^0\subset Y$, whose fibers are equivalence classes with respect to the relation determined by the chains of deformations of the $\tilde{\Gamma}_j$'s, $j\in J$. Let $y\in Y^0$ be a general point, and consider the fiber $\varphi^{-1}(\varphi(y))\subset Y$, which is a proper closed subset because $Y$ is not rationally connected by curves $\tilde{\Gamma}_j$, since its Picard number is $|J|+1$.

Let $V$ be the irreducible component of $\varphi^{-1}(\varphi(y))$ passing by $y$. It is a smooth projective variety satisfying that any curve $\tilde{\Gamma}_j$ meeting it is strictly contained in it. Then the restriction $\pi_{|V}:V\to X$ is surjective, and the proof may be finished by showing that it is an isomorphism.

Suppose this is not the case. Since $Z$ is simply connected, then the branch locus $B\subset X$ of $\pi_{|V}$ is nonempty. Since $X$ is rationally connected by the $\Gamma_j$'s, there exists an index $k\in J$ and a curve $\Gamma_k$ meeting $B$ in a point $b\in B$ and not contained in it. But then the inverse image of a general point $x\in \Gamma_k$ has $\deg(\pi_{|V})$ inverse images, providing $\deg(\pi_{|V})$ sections of the form $\tilde{\Gamma}_k$, over the curve $\Gamma_k$, contained in $V$. We conclude that $b$ has  $\deg(\pi_{|V})$ inverse images as well, a contradiction.
\end{proof}

\subsection{Smoothness of the VMRT's}

Let us consider, for any $i= 1, \dots, 4$, the contraction $\pi^i:X\to X^i$ and the unsplit family of rational curves $\cH_i$ in $X^i$ containing all the curves of the form $\pi^i(\Gamma_i)$. In this section we will show the following:
\begin{proposition}\label{prop:smoothvmrt} With the same notation as above, the VMRT of $\cH_i$ at a general point of $X^i$ is smooth, for any $i=1, \dots, 4$.
\end{proposition}
By the smoothness of $\pi^i$ (Proposition \ref{prop:fibersf4}), $X^i$ is a Fano manifold of Picard number one. Then the statement reduces to check that the ample generator of $\Pic(X^i)$, that we denote by $L_i$, satisfies the hypotheses of Lemma \ref{lem:smoothVMRT}. We will show this in Lemma \ref{lem:fund} and Proposition \ref{prop:spanned} below.

Note first that the determinant of the Cartan matrix of $\mathrm{F}_4$ is $1$, hence there exist line bundles $\Lambda_i \in \Pic(X)$ such that $\Lambda_i \cdot \Gamma_j= \delta_{ij}$, which allows to state the following:

\begin{lemma}\label{lem:fund}
With the same notation as above, $\Lambda_i=(\pi^i)^*L_i$ and, in particular $L_i\cdot \pi^i_*\Gamma_i=1$, for all $i$.
\end{lemma}

\begin{proof}
Since $\pi^i$ is smooth by Proposition \ref{prop:fibersf4},
$\pi^i_*\Lambda_i$ is a line bundle, satisfying $(\pi^i)^*\pi^i_*\Lambda_i=\Lambda_i$. Therefore $\pi^i_*\Lambda_i\cdot\pi^i_*\Gamma_i=\Lambda_i\cdot\Gamma_i=1$,  and so $\pi^i_*\Lambda_i$ is necessarily the ample generator of $\Pic(X_i)$.
\end{proof}

In order to check, finally, the spannedness of $L_i$ we will first show how to construct global sections of $\Lambda_i$ by means of Bott-Samelson varieties.

\begin{lemma}  \label{lem:pullbacks}
With the same notation as above, there exists a list $\ell=(l_1, \dots, l_{23})$ in $I$ such that  $w(\ell)$ is reduced and $f_\ell(Z_\ell)$ is
a divisor in the linear system $|\Lambda_i|$.
\end{lemma}

\begin{proof}
Let $w_0$ be the longest element in the Weyl group of $\mathrm{F}_4$, and let $\ell=(l_1, \dots, l_{23})$ be a list such that  $w(\ell)$ is a reduced expression of $w_0 \circ r_i$. Since, for every $j \not = i$, $\lambda(w_0 \circ r_i \circ r_j) = \lambda(w_0) -2$ (see \cite[10.3 Lemma A]{H}), then $\ell_j=(l_1, \dots, l_{23},j)$ satisfies that $w(\ell_j)$ is not reduced and, by Corollary \ref{cor:stein2}, $\dim f_{\ell_j}(Z_{\ell_j}) < \dim Z_{\ell_j}$. Hence $f_{\ell_j}(Z_{\ell_j})=f_{\ell}(Z_{\ell})$ and, in particular, $f_{\ell}(Z_{\ell})\cdot \Gamma_j=0$.

On the other hand $\ell_i=(l_1, \dots, l_{23},i)$ gives a reduced expression $w(\ell_i)$ of $w_0$, hence
$f_{\ell_i}$ is birational by Corollary \ref{cor:birat} and $f_{\ell}(Z_{\ell})$ is a unisecant divisor for $\pi_i$.
\end{proof}

\begin{proposition}\label{prop:spanned}
With the same notation as above, the line bundle  $L_i$ is spanned by global sections.
\end{proposition}

\begin{proof} Let $\ell$ be as in Lemma \ref{lem:pullbacks}, and let ${\ell}':=(l_{23},\dots,l_1)$. Given any $y\in X^i$, and taking $x\in X$, $y=\pi^i(x)$, we may consider the Bott-Samelson variety $Z_{{\ell}'}$ starting at $Z_{{\ell}'[23]}=\{x\}$. Take any point $x'\in X\setminus f_{{\ell}'}(Z_{{\ell}}')$ and consider the Bott-Samelson variety $Z_\ell$ with $Z_{\ell[23]}=\{x'\}$. Then, by construction, $x\notin f_\ell(Z_\ell)$, so $\pi^i(f_\ell(Z_\ell))$ is a divisor in the linear system $|L_i|$ not passing by $x$. This concludes the proof.
\end{proof}

\subsection{Determination of $X^1$}

We will prove now that $X^1$ is isomorphic to its homogeneous model, by showing first that its VMRT at a general point is the appropriate one. We will start by computing some numerical invariants of $X^1$:

\begin{lemma}\label{lem:dimindex}
With the same notation as above, $X^1$ is a Fano manifold of dimension $15$ and index $8$.
\end{lemma}

\begin{proof}
That $X^1$ is smooth and $15$-dimensional follows from Propositions \ref{prop:dim} and \ref{prop:fibersf4}.
The index of  $X^1$  can then be seen as the minimum integer $k$
such that $H^{15}(X^1,-kL_1) \not = 0$. Since
$$ H^{15}(X^1,-kL_1)=H^{15}(X,-k\Lambda_1)=H^{15}(\overline{X},-k\overline{\Lambda}_1)=H^{15}(\overline{X}^1,-k \overline L_1),$$
where the equality in the middle follows from Corollary \ref{cor:cohomequal}, we get that the index of $X^1$
is equal to the index of $\overline{X}^1$, which is $8$. %\marginpar{\tiny Reference??}
\end{proof}

\begin{proposition}\label{prop:maptovmrt}
With the same notation as above, $X^1 \simeq \overline{X}^1$.
\end{proposition}

\begin{proof} Let us denote by $\G$ the Lagrangian Grassmannian of 3-dimensional  subspaces in $\C^6$ which
are isotropic with respect to a fixed symplectic form, i.e. the rational homogeneous space corresponding
to the Dynkin diagram $\mathrm{C}_3$ marked on the third node.

By Proposition \ref{prop:smoothvmrt}, $\cC_x$ is smooth, for a general $x \in X^1$.
Let $ \bar x\in X$ be such that $\pi^1(\bar x)=x$, and consider the Bott-Samelson variety $Z_\ell$, corresponding to the list $(2,3,4,2,3,4,2,3,4,1)$, constructed starting from the point $\bar x$.
By Corollary \ref{cor:simple} and Proposition \ref{prop:uniqueBC}, we have $Z_\ell \simeq
\P(f_{\ell[1]}^*(K_1)\oplus \cO_{Z_{\ell[1]}})\simeq \P(\overline f_{\ell[1]}^*(\overline K_1)\oplus \cO_{\overline Z_{\ell[1]}})\simeq %%%%% IF WE WANT THE COMPLETE EXPLANATION
\overline Z_{\ell}$.

By the universal property of $\cH_1$, there exists a morphism $h_x: Z_{\ell[1]} \to (\cH_1)_x$ such that $p_{\ell[1]}:Z_\ell \to Z_{\ell[1]}$ is the pull-back of the universal family $\pi_x:\P(\cE_x)\to(\cH_1)_x$. Since $h_x$ contracts curves on which $\cO_{Z_{\ell[1]}} \oplus f^*_{\ell[1]}K_1$ is trivial, it factors through $\psi_1:=\pi^2 \circ {f_{\ell[1]}}$, and we have a commutative diagram:
\[\xymatrix@=30pt{
Z_\ell=\P(\cO_{Z_{\ell[1]}} \oplus f^*_{\ell[1]}K_1)    \ar[r]^(.45){} \ar[d]_{p_{\ell[1]}}& \P(\cO_\G \oplus \cO_\G(1)) \ar[r]\ar[d]& \P(\cE_x) \ar[d]^{\pi_x}  \ar[r]_{\iota}  &X^1  \\
Z_{\ell[1]}  \ar[r]^(.45){\psi_1} & \G \ar[r]^{\psi_2}& (\cH_1)_x &  \\
} \]
Since $\G$ has dimension six and Picard number one, and $\dim \cH_x= -K_{X^1} \cdot \pi^1_*\Gamma_1 -2 = 6$, the morphism $\psi_2$, which is non constant --otherwise $f^*_{\ell[1]}K_1$ would be trivial--, is surjective,
hence the composition $\f:=\tau_x \circ \psi_2:\G \to \cC_x$ is a surjective morphism such that $\f^*\cO_{\cC_x}(1)= c_1( \psi_2^*\cE_x)=\cO_\G(1)$.

By \cite[Main Theorem]{Lau}, $\f$ is an isomorphism unless $\cC_x$ is a projective space. In the latter case, since $\f^*\cO_{\cC_x}(1) = \cO_{\G}(1)$, $\cC_x$ would be a linear space, so, by \cite[Proposition~5]{Hw3}, $X^1$ would be the projective space, contradicting that $X^1$ has dimension $15$ and index $8$.

So we have proved that $\f:\G \to \cC_x$ is an isomorphism, given by a linear projection of $\phi(\G)$, where $\phi$ is the embedding given by the complete linear system $|\cO_{\G}(1)|$.
On the other hand, by \cite[III. 1.4]{Zak}, $\phi(\G)$ cannot be projected isomorphically. This implies that $\cC_x$ is projectively equivalent to $\G$ embedded by  the complete linear system $|\cO_{\G}(1)|$.
We can thus apply \cite[Main Theorem]{HH} (which is a generalization of \cite[Main Theorem]{Mk2}) to get that $X^1 \simeq \overline X^1$.
\end{proof}

\subsection{Reconstructing the complete flag}

In order to prove that $X \simeq \overline X$ we will use the arguments in \cite[Section 6]{MOSW}. Let us note that in the quoted reference $X$ is assumed to be an FT-manifold, but this assumption is used only to show that the contractions of $X$ have the properties described in \cite[Proposition~4]{MOSW}, which we have shown to hold (see section \ref{ssec:smoothcont})  in the case we are considering.

\begin{proof}[of Theorem \ref{conj:CPforFT}, case $\mathrm{F}_4$]
Arguing as in \cite[Proposition~11]{MOSW} we can show that there is  a commutative diagram
\[\xymatrix@=30pt{
X^{1,2}  \ar@/^1pc/[rr]^{{\pi^{\{1,2\},1}} }  \ar[r]_(.45){\tilde{h}} \ar[d]_{\pi^{\{1,2\},2} }& \overline X^{1,2} \ar[d]^{{\overline\pi^{\{1,2\},2}} }  \ar[r]_{{\overline\pi^{\{1,2\},1}} } & \overline X^1   \\
X^{2}  \ar[r]_(.45)h  & \overline X^2 &  \\
} \]
 Since $\dim X^2 = \dim \overline X^2$ and $h$ is not constant, $h$ is a finite surjective map.
By \cite[Main Theorem]{Lau} $\tilde h$ restricted to the fibers of $\pi^{\{1,2\},1}$  is an isomorphism. It follows that $h$ is bijective, hence an isomorphism.
We conclude by applying \cite[Proposition~12]{MOSW} with $I_1=\{1,2\}$.
\end{proof}

\section{Consequences for the Campana-Peternell Conjecture}\label{sec:appCP}

In this section we will show the implications of our methods to the  Campana-Peternell Conjecture (see \ref{conj:CPconj}). Along this section $X$ will denote a complex projective Fano manifold with nef tangent bundle $T_X$, and we will say that $X$ is a CP-manifold. We define the width of $X$ as a measure of how far is $X$ from being a Fano manifold whose elementary contractions are smooth $\P^1$-fibrations.

\begin{definition}
Given a CP-manifold $X$, we define:
$$\tau(X):=\sum_C
(-K_X\cdot C-2)\in\Z_{\geq 0},$$
where the sum is taken over all the classes $C$ of rational curves of minimal ($-K_X$)-degree belonging to extremal rays of $\cNE{X}$.	
\end{definition}

In particular $\tau(X)=0$ if and only if every elementary contraction of $X$ is a smooth $\P^1$-fibration, i.e. Theorem \ref{conj:CPforFT} can be read as:

\begin{corollary}\label{cor:CP2}
Any CP-manifold $X$ with $\tau(X)=0$
is isomorphic to the quotient of a semisimple group $G$ by its Borel subgroup $B$.
\end{corollary}

\begin{proof}
Assume that $\tau(X)=0$, and let $\pi:X\to X_i$ be an elementary contraction, associated to an extremal ray $R_i$ generated by the class of a minimal rational curve $\Gamma_i$. Let $p:\cU\to \cM$ be the family of deformations of $\Gamma_i$, with evaluation morphism $q$. Since by hypothesis $-K_X\cdot\Gamma_i=2$, then $q:\cU\to X$ is finite. Moreover, following the proof of \cite[Proposition~4]{HM}, $q$ is a holomorphic submersion. Now, since $X$ is simply connected, it follows that $q$ is an isomorphism, and that the contraction $\pi$ is a $\P^1$-fibration.
\end{proof}

Then Conjecture \ref{conj:CPconj} boils down to proving the following:

\begin{conjecture}\label{conj:CP1}
Given a CP-manifold satisfying $\tau(X)>0$, there exists a contraction $f:X'\to X$ from a CP-manifold $X'$ satisfying $\tau(X')<\tau(X)$.
\end{conjecture}

%% If you have bibdatabase file and want bibtex to generate the
%% bibitems, please use
%%
%%  \bibliographystyle{elsarticle-harv}
%%  \bibliography{<your bibdatabase>}

%% else use the following coding to input the bibitems directly in the
%% TeX file.

\end{document}